\documentclass[a4paper,UKenglish]{lipics-v2021}
\title{Complexity Framework for Forbidden Subgraphs IV: The Steiner Forest Problem}

\author{Hans L. Bodlaender}{Utrecht University, Utrecht, The Netherlands}{h.l.bodlaender@uu.nl}{}{}
\author{Matthew Johnson}{Durham University, Durham, United Kingdom}{matthew.johnson2@durham.ac.uk}{}{}
\author{Barnaby Martin}{Durham University, Durham, United Kingdom}{barnaby.d.martin@durham.ac.uk}{}{}
\author{Jelle J. Oostveen}{Utrecht University, Utrecht, The Netherlands}{j.j.oostveen@uu.nl}{0009-0009-4419-3143}{}
\author{Sukanya Pandey}{Utrecht University, Utrecht, The Netherlands}{s.pandey1@uu.nl}{0000-0001-5728-1120}{}
\author{Dani\"el Paulusma}{Durham University, Durham, United Kingdom}{daniel.paulusma@durham.ac.uk}{0000-0001-5945-9287}{}
\author{Siani Smith}{University of Bristol and Heilbronn Institute for Mathematical Research, Bristol, United Kingdom}{siani.smith@bristol.ac.uk}{}{}
\author{Erik Jan van Leeuwen}{Utrecht University, Utrecht, The Netherlands}{e.j.vanleeuwen@uu.nl}{0000-0001-5240-7257}{}

\authorrunning{H.L. Bodlaender et al.}
\Copyright{Hans L. Bodlaender, Matthew Johnson, Barnaby Martin, Jelle J. Oostveen, Sukanya Pandey, Dani\"el Paulusma, Siani Smith, Erik Jan van Leeuwen}


\nolinenumbers

\hideLIPIcs

\EventEditors{John Q. Open and Joan R. Access}
\EventNoEds{2}
\EventLongTitle{42nd Conference on Very Important Topics (CVIT 2016)}
\EventShortTitle{CVIT 2016}
\EventAcronym{CVIT}
\EventYear{2016}
\EventDate{December 24--27, 2016}
\EventLocation{Little Whinging, United Kingdom}
\EventLogo{}
\SeriesVolume{42}
\ArticleNo{23}

\usepackage{microtype,todonotes}
\usepackage{enumerate}
\usepackage{boxedminipage}
\usepackage{tikz,tkz-graph}
\usepackage{hyperref}
\usepackage{amssymb,amsmath}
\usepackage{slashbox}
\usepackage{xspace}

\newtheorem{open}{Open Problem}
\newcounter{ctrclaim}[theorem]
\newcounter{ctrcase}[theorem]

\DeclareMathOperator{\vc}{vc}

\newcommand{\undC}{C}

\newcommand{\NP}{{\sf NP}}

\newcommand{\ssi}{\subseteq_i}

\newcommand{\sft}{\mathsf{sf}}
\newcommand{\arch}{archipelago\xspace}
\newcommand{\archs}{archipelagos\xspace}
\newcommand{\refines}{\sqsubseteq}

\newcommand{\FPT}{{\sf FPT}}

\newcommand{\problemdef}[3]{
        \begin{center}
                \begin{boxedminipage}{.99\textwidth}
                        \textsc{{#1}}\\[2pt]
                        \begin{tabular}{ r p{0.8\textwidth}}
                                \textit{~~~~Instance:} & {#2}\\
                                \textit{Question:} & {#3}
                        \end{tabular}
                \end{boxedminipage}
        \end{center}
}

\begin{document}
\maketitle

\begin{abstract}
We study {\sc Steiner Forest} on $H$-subgraph-free graphs, that is, graphs that do not contain some fixed graph~$H$ as a (not necessarily induced) subgraph. We are motivated by a recent framework that completely characterizes the complexity of many problems on $H$-subgraph-free graphs. However, in contrast to e.g.\ the related {\sc Steiner Tree} problem, {\sc Steiner Forest} falls outside this framework. Hence, the complexity of {\sc Steiner Forest} on $H$-subgraph-free graphs remained tantalizingly open.

In this paper, we make significant progress towards determining the complexity of {\sc Steiner Forest} on $H$-subgraph-free graphs. Our main results are four novel polynomial-time algorithms for different excluded graphs $H$ that are central to further understand its complexity.  Along the way, we study the complexity of {\sc Steiner Forest} for graphs with a small $c$-deletion set, that is, a small set $S$ of vertices such that each component of $G-S$ has size at most~$c$. Using this parameter, we give two noteworthy algorithms that we later employ as subroutines. First, we prove that {\sc Steiner Forest} is 
fixed-parameter tractable
by $|S|$ when $c=1$ (i.e.\ the vertex cover number). Second, we prove that {\sc Steiner Forest} is polynomial-time solvable for graphs with a $2$-deletion set of size at most~$2$. 
The latter result is tight, as the problem is \NP-complete for graphs with a $3$-deletion set of size~$2$. 
\end{abstract}

\section{Introduction}\label{s-intro}

In this paper, we consider the complexity of a classical graph problem, {\sc Steiner Forest}, restricted to graphs that do not contain some fixed graph~$H$ as a subgraph. Such graphs are said to be {\it $H$-subgraph-free}, that is, they cannot be modified to $H$ by a sequence of edge deletions and vertex deletions. A graph $G$ is {\it $H$-free} if $G$ cannot be modified into $H$ by a sequence of vertex deletions only. Even though $H$-free graphs
are more widely studied in the literature, $H$-subgraph-free graphs are also highly interesting, as was recently shown through the introduction of a large, general framework for the subgraph relation~\cite{JMOPPSV,JMPPSvL22,JMPPSV,MPPSV}.

Before giving our results on {\sc Steiner Forest}, we first explain the framework.
For a set of~graphs ${\cal H}$, a graph $G$ is  {\it ${\cal H}$-subgraph-free}  if $G$ is $H$-subgraph-free for every $H\in {\cal H}$.
In order to unify and extend known classifications for {\sc Independent Set}~\cite{AK92}, {\sc Dominating Set}~\cite{AK92}, {\sc List Colouring}~\cite{GP14}, {\sc Long Path}~\cite{AK92} and {\sc Max-Cut}~\cite{Ka12} on ${\cal H}$-subgraph-free graphs (for finite ${\cal H}$), a systematic approach was followed in~\cite{JMOPPSV}.
A class of graphs has bounded {\it treewidth} if there exists a constant~$c$ such that every graph in it has treewidth at most~$c$.  
For an integer $k\geq 1$, the {\it $k$-subdivision} of an edge $e=uv$ of a graph replaces $e$ by a path of length $k+1$ with endpoints $u$ and $v$ (and $k$ new vertices). 
The {\it $k$-subdivision} of a graph~$G$ is the graph obtained from $G$ after $k$-subdividing each edge. 
For a graph class ${\cal G}$ and an integer~$k$, let ${\cal G}^k$ consist of the $k$-subdivisions of the graphs in ${\cal G}$.
A graph problem $\Pi$ is \NP-complete {\it under edge subdivision of subcubic graphs} if for every integer $j \geq 1$, there is an integer~$\ell \geq j$ such that:
if $\Pi$ is computationally hard for the class ${\cal G}$ of subcubic graphs (graphs with maximum degree at most~$3$), then $\Pi$ is computationally hard for ${\cal G}^{\ell}$.
Now, $\Pi$ is a {\it C123-problem} if:\\[-8pt]
\begin{enumerate}
\item [{\bf C1.}] $\Pi$ is polynomial-time solvable for every graph class of bounded treewidth,
\item [{\bf C2.}] $\Pi$ is \NP-complete for the class of subcubic graphs, and
\item [{\bf C3.}] $\Pi$ is \NP-complete under edge subdivision of subcubic graphs.\\[-8pt]
\end{enumerate}
\noindent
A {\it subdivided} claw is a graph obtained from a {\it claw} ($4$-vertex star) by subdividing each of its three edges zero or more times. The {\it disjoint union} of two vertex-disjoint graphs $G_1$ and $G_2$ is graph $(V(G_1)\cup V(G_2), E(G_1)\cup E(G_2))$. The set ${\cal S}$ consists of all graphs that are disjoint unions of subdivided claws and paths. 
We can now state the complexity classification of~\cite{JMOPPSV}.

\begin{theorem}[\cite{JMOPPSV}]\label{t-dicho2}
Let $\Pi$ be a C123-problem. For a finite set ${\cal H}$, the problem $\Pi$ on ${\cal H}$-subgraph-free graphs is polynomial-time solvable if ${\cal H}$ contains a graph from ${\cal S}$ (or equivalently, if the class of ${\cal H}$-subgraph-free graphs has bounded treewidth) and \NP-complete otherwise.
\end{theorem}

\noindent
Examples of C123-problems include {\sc Independent Set}, {\sc Dominating Set}, {\sc Long Path}, {\sc Max Cut}, {\sc List Colouring}, 
 {\sc Disjoint Paths},  {\sc Odd Cycle Transversal}, {\sc Perfect Matching Cut}, {\sc Steiner Tree}, and many more; see a (long) table of problems in~\cite{JMOPPSV}. 
 
 Nevertheless, there are many well-known graph problems that are not C123. For example, {\sc Colouring}~\cite{Br41}, {\sc Connected Vertex Cover}~\cite{UKG88}, {\sc Feedback Vertex Set}~\cite{UKG88}, 
 {\sc Independent Feedback Vertex Set}~\cite{JMPPSV} and {\sc Matching Cut}~\cite{Ch84}  do not satisfy C2, whereas   {\sc Hamilton Path} and {\sc $k$-Induced Disjoint Paths} do not satisfy C3~\cite{MPPSV}.

There are also problems that only satisfy C2 and C3 but not C1. For example, {\sc Subgraph Isomorphism} is \NP-complete even for input pairs of path-width~$1$. A few years ago, Bodlaender et al.~\cite{BHKKOO20} settled the complexity of {\sc Subgraph Isomorphism} for $H$-subgraph-free graphs for connected graphs~$H$ except $H=P_5$. For disconnected graphs $H$, they made significant progress and reduced all open cases to $H = P_5$ and $H=2P_5$. 
This shows that the following question is challenging: 

\medskip
\noindent
{\it How do \emph{C23-problems}, i.e., that satisfy C2 and C3 but not C1, behave for $H$-subgraph-free graphs? Can we still classify their computational complexity?}

\medskip
\noindent
We consider this question for {\sc Steiner Forest}.
A {\it Steiner forest} of a graph $G$, with a set $S=\{(s_1,t_1),\ldots,(s_p,t_p)\}$ of specified pairs of vertices called {\it terminals},  is a subgraph $F$ of $G$, such that $s_i$ and $t_i$, for every $i\in \{1,\ldots,p\}$, belong to the same connected component of $F$. 

\problemdef{Steiner Forest}{A graph $G$, a set $S$ of terminal pairs and an integer $k$.}{Does $(G,S)$ have a Steiner forest $F$ with
$|E(F)|\leq k$?}

\noindent
In our problem definition, we consider unweighted graphs and the goal is to find Steiner forests with a small number of {\it edges}. Moreover, {\sc Steiner Forest} generalizes the C123-problem {\sc Steiner Tree}, which is to decide whether for a given integer~$k$, a graph $G$ with some specified set $S$ of vertices has a tree~$T$ with $|E(T)|\leq k$ containing every vertex of $S$: take all pairs of vertices of $S$ as terminal pairs to obtain an equivalent instance of {\sc Steiner Forest}.

For a constant~$c$, a {\it $c$-deletion set} of a graph $G=(V,E)$ is a set $T\subseteq V$ such that each connected component of $G-T$ has size at most~$c$. The {\it $c$-deletion set number} of $G$ is the size of a smallest $c$-deletion set; see e.g.~\cite{BES87,BDKOP22,DEGKO17,FF18,FMS16} for results on this parameter and related ones, such as the vertex integrity, safe number and fracture number. 
Bateni, Hajiaghayi and Marx~\cite{BHM11} proved the following (see also Section~\ref{s-npc}).\footnote{Previously, \NP-completeness for treewidth~$3$ was known for 
 {\sc Weighted Steiner Forest}~\cite{Ga10}.}
  
\begin{theorem}[\cite{BHM11}]\label{t-tw}
{\sc Steiner Forest} is polynomial-time solvable for graphs of treewidth at most~$2$, but \NP-complete for graphs of treewidth~$3$, tree-depth~$4$, and $3$-deletion set number~$2$.
\end{theorem}

\noindent
This shows in particular that {\sc Steiner Forest} does not satisfy C1, unlike {\sc Steiner Tree}~\cite{ALS91}. 
    As {\sc Steiner Tree} satisfies C2 and C3~\cite{JMOPPSV}, {\sc Steiner Forest} satisfies C2 and C3 and is thus a C23-problem, unlike {\sc Steiner Tree} which is  C123~\cite{JMOPPSV}. This leaves the complexity of {\sc Steiner Forest} on $H$-subgraph-free graphs tantalizingly open.

\subsection*{Our Results}

Recall that ${\cal S}$ is the class of disjoint unions of paths and subdivided claws. 
For positive integers $a$ and $b$, we use $K_{a,b}$ to denote the complete bipartite graph with $a$ vertices on one side and $b$ on the other. We use $S_{a,b,c}$ to denote the graph obtained from the claw ($K_{1,3}$) by subdividing the three edges $a-1$, $b-1$, and $c-1$ times respectively. We use $P_a$ to denote the path on $a$ vertices.
See Figure~\ref{fig:smallgraphs} for some examples of these graphs.
For two graphs $H_1$ and $H_2$, we write $H_1\subseteq H_2$ if $H_1$ is a subgraph of $H_2$, i.e., $V(H_1)\subseteq V(H_2)$~and $E(H_1)\subseteq E(H_2)$. We write $H_1+H_2$ to mean the disjoint union of $H_1$ and $H_2$ and $sH_1$ to denote the disjoint union of $s$ copies of $H_1$. 
Our results on {\sc Steiner Forest} for $H$-subgraph-free graphs are:

\begin{figure}
\begin{center}
	\includegraphics[scale=0.7]{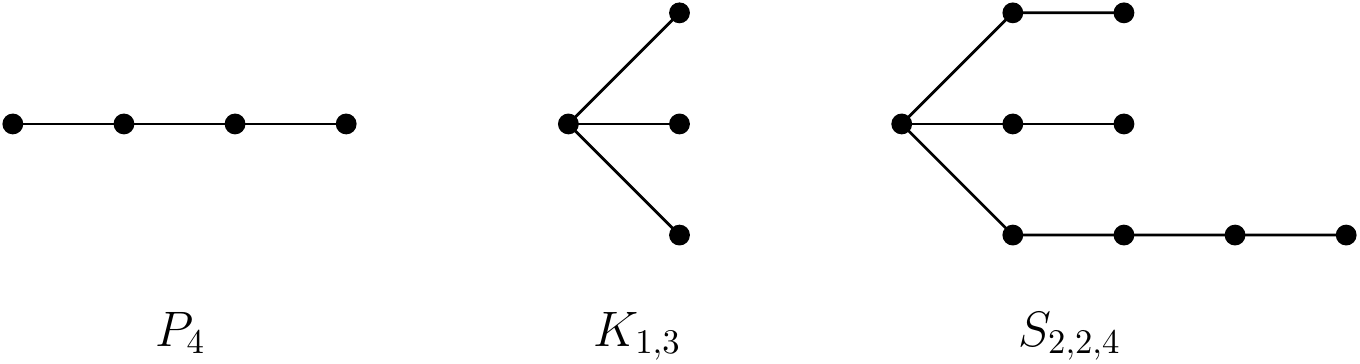}
\end{center}
\caption{Illustration of the graphs $P_4$, $K_{1,3}$, and $S_{2,2,4}$.}
\label{fig:smallgraphs}
\end{figure}

\begin{theorem}\label{t-main}
For a graph~$H$, {\sc Steiner Forest} on $H$-subgraph-free graphs is\\[-5mm]
\begin{itemize}
\item polynomial-time solvable if  $H\subseteq 2K_{1,3}+P_3+sP_2, 2P_4+P_3+sP_2, P_9+sP_2$ or $S_{1,1,4}+sP_2$ for each $s\geq 0$, and
\item \NP-complete if $H\supseteq 3K_{1,3}, 2K_{1,3}+P_4, K_{1,3}+2P_4,3P_4$ or if $H\notin {\cal S}$.
\end{itemize}
\end{theorem}

\noindent
The gap between the two cases could be significantly reduced if we could resolve an intriguing question about the parameterized complexity of {\sc Steiner Forest}
(see Section~\ref{s-con} for a detailed discussion).

As graphs of tree-depth~$3$ are $P_8$-subgraph-free, Theorems~\ref{t-tw} and~\ref{t-main} yield a dichotomy:

\begin{corollary}\label{c-d}
For a constant~$t$, {\sc Steiner Forest} on graphs of tree-depth~$t$ is polynomial-time solvable if $t\leq 3$ and \NP-complete if $t\geq 4$.
\end{corollary}

\noindent
The \NP-hardness part of Theorem~\ref{t-main} follows from the gadget from Theorem~\ref{t-tw} and \NP-completeness of {\sc Steiner Forest} when $H\notin {\cal S}$~\cite{BBJPPL21}, as shown in Section~\ref{s-npc}.
For the polynomial part, we first make some observations on $2$-connectivity and high-degree vertices in Section~\ref{s-obs}.
Then, in Section~\ref{s-c}, we show that {\sc Steiner Forest} is \FPT\ with respect to the {\it vertex cover number} ($1$-deletion set number) of a graph, and there we also prove the polynomial-part of the following dichotomy, in which the \NP-completeness part is due to Theorem~\ref{t-tw} \cite{BHM11}.

\begin{theorem}\label{t-2d2}
For a constant~$c$, {\sc Steiner Forest} on graphs with a $c$-deletion set of size at most~$2$  is polynomial-time solvable if $c\leq 2$ and \NP-complete if $c\geq 3$.
\end{theorem}

\noindent
Using the algorithms from Sections~\ref{s-obs} and~\ref{s-c} as subroutines, we prove the polynomial part of Theorem~\ref{t-main} in Section~\ref{s-poly}. 

Finally, a graph parameter $p$ {\it dominates} a graph parameter~$q$ if there is a function~$f$ such that $p(G)\leq f(q(G))$ for every graph~$G$.
If $p$ dominates $q$ but $q$ does not dominate $p$, then $p$ is {\it more powerful} than $q$.   
From the definitions, it follows that treewidth is more powerful
than treedepth, which is more powerful than $c$-deletion set number (fixed $c\geq 2$), which is more powerful than the vertex cover number ($1$-deletion set number).
Given the hardness results in Theorems~\ref{t-tw}, Corollary~\ref{c-d} and Theorem~\ref{t-2d2}, this
also gives an indication to what extent our new \FPT\ result for vertex cover number is best possible.

\section{NP-Completeness Results}\label{s-npc}

For a rooted forest~$F$, the  \emph{closure} $\undC(F)$ is the graph with vertex set $V(F)$  with the property that two vertices $u$ and $v$ are adjacent in $\undC(F)$ if and only if $u$ is an ancestor of $v$ in $F$. The {\it depth} of $F$ (number of ``layers'') is equal to the height of $F$ plus~$1$. We say that $F$ is a {\it tree-depth decomposition} of a graph $G$ if $G$ is a subgraph of $\undC(F)$. The \emph{tree-depth} of $G$ is the minimum depth over all tree-depth decompositions of $G$.

Bateni, Hajiaghayi and Marx~\cite{BHM11} explicitly proved that {\sc Steiner Forest} is \NP-complete for graphs of treewidth~$3$, see also Theorem~\ref{t-tw}. 
The additional properties in the lemma below can be easily verified from inspecting their gadget, which is displayed in Figure~\ref{f-gs}.

\begin{figure}[t!]
\centering
\includegraphics[scale=0.82]{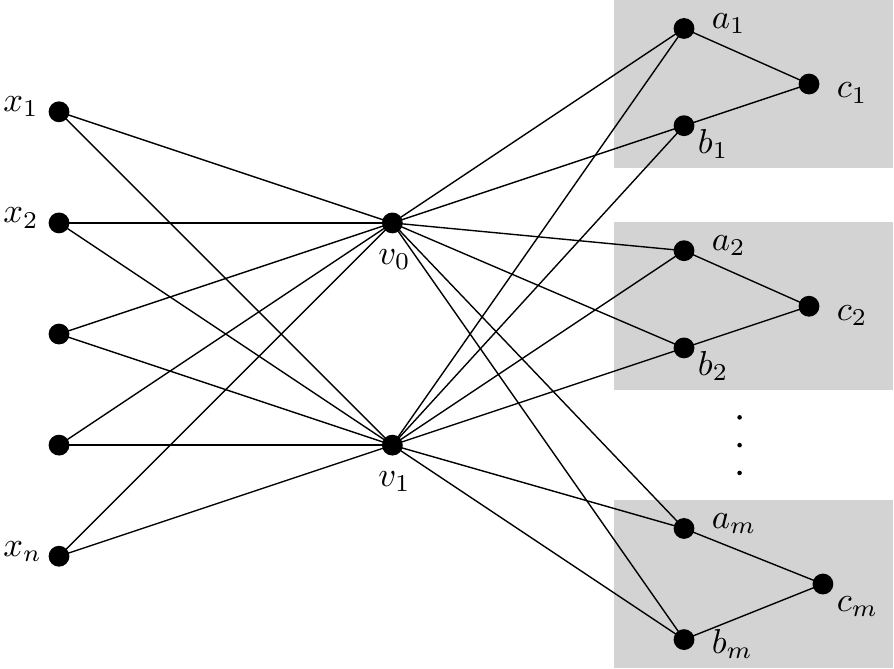}
\caption{The graph $G$ (gadget from~\cite{BHM11}) used in the proof of Lemma~\ref{l-np}.}\label{f-gs}\label{fig:gadgets}
\vspace*{-1mm}
\end{figure}

\begin{lemma}\label{l-np}
{\sc Steiner Forest} is \NP-complete for $(3K_{1,3},2K_{1,3}+P_4,K_{1,3}+2P_4,3P_4)$-subgraph-free graphs of tree-depth~$4$ with $3$-deletion set number~$2$.
\end{lemma}

\begin{proof}
For a given a so-called R-formula $\phi$ with $n$ variables and $m$ clauses,  Bateni, Hajiaghayi and Marx~\cite{BHM11} constructed an instance $(G,S)$ of {\sc Steiner Forest} such that $\phi$ is satisfiable if and only if $(G,S)$ has a solution of $n+3m$ edges. The graph $G$ (which has treewidth~$3$) is illustrated in Figure~\ref{f-gs}. It remains to make three observations. First, every $P_4$ or $K_{1,3}$ in $G$ needs to use either $v_0$ or $v_1$. Second, the set $\{v_0,v_1\}$ is a $3$-deletion set.  Third, $G$ has tree-depth~$4$; take $\{v_0\}$ as layer~0; $\{v_1\}$ as layer~$1$; the set of all $c$-type vertices and $x$-type vertices as layer~$2$; and the set of all $a$-type vertices and $b$-type vertices as layer~$3$.
\end{proof}

\noindent
{\it Proof of \NP-completeness part of Theorem~\ref{t-main}.} 
It is known that {\sc Steiner Tree}, and thus {\sc Steiner Forest}, is \NP-complete for $H$-subgraph-free graphs if $H\notin {\cal S}$~\cite{BBJPPL21}.
The \NP-completeness part of Theorem~\ref{t-main} now follows immediately from this observation and Lemma~\ref{l-np}.

\section{Basic Polynomial Results}\label{s-obs}

A {\it minimum} Steiner forest for an instance $(G,S)$ is one with the smallest number of edges. We denote the number of edges of such a forest by $\sft(G,S)$. We always assume that for any terminal pair $(s,t) \in S$, $s$ and $t$ are distinct (as any pair where $s=t$ can be removed without affecting the feasibility or the size of a minimum solution). A vertex $v$ is a \emph{terminal vertex} if there is a pair $(s,t) \in S$ with $v=s$ or $v=t$.

We now describe several important subroutines for our later algorithms. 
A \emph{cut vertex} of a connected graph is a vertex whose removal yields a graph with at least two connected components. A graph without a cut vertex is \emph{$2$-connected}. A \emph{$2$-connected component} or \emph{block} of a graph is a maximal subgraph that is $2$-connected. 
A graph class is \emph{hereditary} if it is closed under deleting vertices.

\begin{lemma}\label{l-2con}
For every hereditary graph class ${\cal G}$, if {\sc Steiner Forest} is polynomial-time solvable for the subclass of $2$-connected graphs of ${\cal G}$, then it is polynomial-time solvable for ${\cal G}$.
\end{lemma}

\begin{proof}
Let $G$ be any graph in ${\cal G}$ and $S \subseteq V(G) \times V(G)$ be any set of terminal pairs. If $G$ is not connected, then we consider each connected component individually and for each individual instance, we restrict the set of terminal pairs to pairs of vertices of that component. If a terminal pair is split among two distinct components, the instance is obviously infeasible.

We may thus assume that the graph $G$ is connected. We use the following recursive algorithm. Let $A$ be the assumed algorithm for {\sc Steiner Forest} on $2$-connected graphs of ${\cal G}$. If $G$ is $2$-connected, then we may simply call $A$ on $(G,S)$. Otherwise, $G$ has a $2$-connected component that contains only a single cut-vertex of $G$ (i.e.\ any leaf of the block-cut tree of $G$). Let $G_1$ be any such component and let $v$ be the cut vertex such that $G_1-v$ is a connected component of $G-v$. Let $G_2 = G-(V(G_1)\setminus\{v\})$ be the remainder of $G$. We now construct two sets of terminal pairs. For $i=1,2$, let $S_i \subseteq V(G_i) \times V(G_i)$ initially be the set of all terminal pairs in $G_i$ for which both terminals are in $V(G_i)$. Then add to $S_i$ the following pairs: for each $(s_j,t_j) \in S$ such that $s_j \in V(G_i)$ and $t_j \not\in V(G_i)$, add the pair $(s_j,v)$ to $S_i$. By definition, $G_1$ is $2$-connected. Apply $A$ to $(G_1,S_1)$. Recurse on $(G_2,S_2)$.

It is immediate that this algorithm runs in polynomial time. Indeed, running $A$ on $G_1$ takes polynomial time. As $|V(G_2)| < |V(G)|$, we apply $A$ at most $|V(G)|$ times on graphs not larger than $G$ and terminal sets not larger than $S$.

To see correctness, let $F$ be a minimum Steiner forest for $(G,S)$. We claim that for $i=1,2$, $F_i := (V(F) \cap V(G_i), E(F) \cap E(G_i))$ is a solution to $(G_i,S_i)$. Indeed, for a pair $(s_j,t_j) \in S$ for which $s_j,t_j \in V(G_i)$, let $P_j$ be an $s_j$-$t_j$ path in $F$. Since $v$ is a cut vertex, $P_j$ cannot contain any edges of $E(G)-E(G_i)$. Hence, $P_j$ exists in $F_i$. For a pair $(s_j,t_j) \in S$ for which $s_j \in V(G_i)$ and $t_j \not\in V(G_i)$, let $Q_j$ be an $s_j$-$t_j$ path in $F$. As $v$ is a cut vertex, $Q_j$ must contain $v$ and $E(Q_j) \cap E(G_i) \subseteq E(F_i)$ is a $s_j$-$v$ path in $G_i$. Hence, $F_i$ is a solution to $(G_i,S_i)$ and thus $\sft(G,S) \geq \sft(G_1,S_1) + \sft(G_2,S_2)$.

Conversely, for $i=1,2$, let $F_i$ be a minimum Steiner forest for $(G_i,S_i)$. We claim that $F := (V(F_1) \cup V(F_2), E(F_1) \cup E(F_2))$ is a solution to $(G,S)$. Consider any terminal pair $(s_j,t_j) \in S$. If $s_j,t_j \in V(G_i)$ for some $i\in\{1,2\}$, then $F_i$ contains an $s_j$-$t_j$ path and thus so does $F$. Otherwise, without loss of generality, $S_1$ contains the pair $(s_j,v)$ and $S_2$ contains the pair $(v,t_j)$ by construction. Hence, $F_1$ contains an $s_j$-$v$ path and $F_2$ contains a $v$-$t_j$ path, and $F$ contains the union of these two paths and thus a $s_j$-$t_j$ path. Therefore, $F$ is indeed a solution to $(G,S)$ and $\sft(G,S) = \sft(G_1,S_1) + \sft(G_2,S_2)$.
\end{proof}

\noindent
The \emph{contraction} of an edge $e = (u,v)$ in a graph~$G$ replaces $u$ and $v$ by a new vertex~$w$ that is adjacent to all former neighbours of $u$ and $v$ in $G$. 
 A \emph{$2$-path} in~$G$ is a path whose internal vertices have degree exactly~$2$ in~$G$. A $2$-path is \emph{maximal} if both its ends have degree not equal to~$2$ in $G$ (so they have degree~$1$ or at least~$3$).
 In particular, the ends may have degree~$1$ or at least~$3$ in $G$.
 
 For a $2$-path~$P$ in~$G$, the following branching lemma considers the case that: 1) the whole path $P$ is in the solution; and 2) parts of $P$ are in the solution. In the latter case, we only distinguish between the parts of $P$ in the solution that are incident on the ends of $P$ and the remaining `middle' part, and branch into relevant cases for each possible pair of edges splitting $P$ in this manner. The lemma shows how to compute these branches, and that they are sufficient to solve the original problem.

\begin{lemma}\label{l-path}
Let $(G,S)$ be an instance of {\sc Steiner Forest}. Let $P$ be a $2$-path in $G$. Let $G_1$ be obtained from $G$ by contracting all edges of $P$ and let $G_2$ be obtained from $G$ by removing all edges of $P$ and any resulting isolated vertices. Then we can compute in polynomial time set $S_1 \subseteq V(G_1) \times V(G_1)$ and integer $w_1$, and sets $S_2^{e,f} \subseteq V(G_2) \times V(G_2)$ and integers $w_2^{e,f}$ for $e,f \in E(P)$ such that $$\sft(G,S) = \min\{\sft(G_1,S_1) + w_1, \min_{e,f \in E(P)} \{\sft(G_2,S_2^{e,f}) + w_2^{e,f}\}\}.$$ 
\end{lemma}

\begin{proof}
We first show how to compute $S_1,S_2$ and $w_1,w_2$. Let $u,v$ be the ends of $P$ and let $x$ be the super-vertex in $G_1$ that is obtained after the contraction of $P$. To obtain $S_1$, start with $S_1 = S$ and remove any terminal pair $(s,t) \in S_1$ for which $s,t \in V(P)$. Note that $V(P)$ includes the ends of $P$. Then, for any terminal $(s,t) \in S_1$ for which $s \in V(P)$ and $t \not\in V(P)$, replace $s$ by $x$. Finally, let $w_1 = |E(P)|$.

Now let $e,f \in E(P)$. We explicitly allow $e=f$. Without loss of generality, $e$ is closer (on $P$) to $u$ than $f$ or $e=f$. Let $u'$ be the endpoint of $e$ closer (on $P$) to $u$ and let $v'$ be the endpoint of $f$ closer (on $P$) to $v$. Note that possibly $u=u'$ or $v=v'$. To compute $w_2^{e,f}$, we proceed as follows. We will compute a graph $G^{e,f}$ that will be a subpath of $P$. If $e=f$, let $G^{e,f}$ be the empty graph. Otherwise, let $G^{e,f}$ be the subpath of $P$ starting at the endpoint of $e$ that is not $u'$ and ending at the endpoint of $f$ that is not $v'$. Let $P^{e,f}$ and $Q^{e,f}$ be the two subpaths of $P-V(G^{e,f})$. Note that these subpaths both exclude $e$ and $f$, as does $G^{e,f}$. If there exists a terminal pair $(s,t) \in S$ for which $s \in V(G^{e,f})$ and $t \not\in V(G^{e,f})$, then set $S^{e,f} = \emptyset$ and $w_2^{e,f} = \infty$. Otherwise, let $S^{e,f}$ be the set of terminal pairs $(s,t) \in S$ for which $s,t \in V(G^{e,f})$ and let $w_2^{e,f} = \sft(G^{e,f},S^{e,f}) + |E(P^{e,f})| + |E(Q^{e,f})|$. Since $G^{e,f}$ is a path, $\sft(G^{e,f},S^{e,f})$ can be computed in polynomial time via Theorem~\ref{t-tw}. Finally, to obtain $S_2^{e,f}$, start with $S_2^{e,f} = S$ and remove any terminal pair $(s,t) \in S_2^{e,f}$ for which one of $s,t \in V(G^{e,f})$. Then, for any terminal $(s,t) \in S_2^{e,f}$ for which $s \in V(P^{e,f}) \cup V(Q^{e,f})$, replace $s$ by $u$ if $s \in V(P^{e,f})$ and by $v$ otherwise. Then do the same with respect to $t$.

We now prove the lemma statement for the constructed $S_1,S_2$ and $w_1,w_2$. Let $F$ be a minimum Steiner forest for $(G,S)$. Suppose that $E(P) \subseteq E(F)$. Then let $F' = ((V(F) - V(P)) \cup \{x\}, E(F) \setminus E(P))$. Clearly, $F'$ is a Steiner forest for $(G_1,S_1)$: edge contractions in a forest cannot create a cycle and any two terminals connected via $P$ are now connected via $x$. Hence, $\sft(G,S) \geq \sft(G_1,S_1) + w_1$. Suppose that $E(P) \not\subseteq E(F)$. Let $e'$ be the edge of $P$ that is not in $F$ and is closest (on $P$) to $u$. Let $f'$ be the edge of $P$ that is not in $F$ and is closest (on $P$) to $v$. Note that $e',f'$ are properly defined and possibly $e'=f'$. By definition, $E(P^{e',f'}) \subseteq E(F)$ and $E(Q^{e',f'}) \subseteq E(F)$. Moreover, we can assume that $|E(F) - E(P^{e',f'}) - E(Q^{e',f'})| = \sft(G^{e,f},S^{e,f})$, or we could replace $E(F) - E(P^{e',f'}) - E(Q^{e',f'})$ by any minimum Steiner forest for $(G^{e,f},S^{e,f})$ to obtain a smaller Steiner forest for $(G,S)$, a contradiction. It follows that $\sft(G,S) \geq \sft(G_2^{e,f},S_2^{e,f}) + w_2^{e,f}$.

Conversely, suppose that the minimum is achieved by $\sft(G_1,S_1)+w_1$ and let $F_1$ be a minimum Steiner forest for $(G_1,S_1)$. Let $F = (F_1 - x) \cup P = ((V(F_1)\setminus\{x\}) \cup V(P), E(F_1) \cup E(P))$; we only remove $x$ if $x \in V(F_1)$, which is not necessarily the case. Observe that $F$ is indeed a forest. Moreover, $F$ is a Steiner forest for $(G,S)$, because any terminal pair with both terminals in $P$ is satisfied by $P$, any terminal pair with exactly one terminal in $P$ is satisfied due to $x$ being the replacement of the terminal by $x$ in the construction of $S_1$, and for any other terminal pair, any path between them that uses $x$ can be expanded using $P$. As $F_1$ and $P$ are edge-disjoint, $\sft(G,S) \leq \sft(G_1,S_1)+w_1$.

Suppose that the minimum is achieved by $e',f' \in E(P)$ and $\sft(G_2^{e',f'},S_2^{e',f'}) + w_2^{e',f'}$ and let $F_2$ be a minimum Steiner forest for $(G_2^{e',f'},S_2^{e',f'})$. Let $F'$ be a minimum Steiner forest for $G^{e',f'}$, which can be found in polynomial time via Theorem~\ref{t-tw}. Let $F = F_2 \cup F' \cup P^{e',f'} \cup Q^{e',f'} = (V(F_2) \cup V(F') \cup V(P^{e',f'}) \cup V(Q^{e',f'}), E(F_2) \cup E(F') \cup E(P^{e',f'}) \cup E(Q^{e',f'}))$. Observe that $F$ is indeed a forest. Moreover, $F$ is a Steiner forest for $(G,S)$, because any terminal pair with both terminals in $G^{e',f'}$ is satisfied by $F'$, any terminal pair with both terminals in $P^{e',f'}$ ($Q^{e',f'}$) is satisfied by $P^{e',f'}$ ($Q^{e',f'}$), any terminal pair with exactly one terminal in $P^{e',f'}$ ($Q^{e',f'}$) is satisfied by $u$ ($v$) being the replacement terminal and $P^{e',f'} \cup F_2$ ($P^{e',f'} \cup F_2$), any terminal pair with exactly one terminal in $P^{e',f'}$ and exactly one in $Q^{e',f'}$ is satisfied by the combination of both, and any other terminal pair is satisfied by $F_2$. As $F_2$, $F'$, $P^{e',f'}$, and $Q^{e',f'}$ are edge-disjoint, $\sft(G,S) \leq \sft(G_2,S_2^{e',f'}) + w_2^{e',f'}$.

From the above, $\sft(G,S) = \min\{\sft(G_1,S_1) + w_1, \min_{e,f \in E(P)} \{\sft(G_2,S_2^{e,f}) + w_2^{e,f}\}\}$.
\end{proof}

\noindent
Using the above lemma, we can obtain the following result.

\begin{lemma}\label{l-degree}
{\sc Steiner Forest} is polynomial-time solvable for graphs with a bounded number of maximal $2$-paths between vertices of degree more than~$2$.
\end{lemma}

\begin{proof}
Let $(G_0,S_0)$ be an instance of {\sc Steiner Forest} where $G_0$ has $k_0$ maximal $2$-paths between vertices of degree more than~$2$. By Lemma~\ref{l-2con}, we may assume that $G_0$ has no vertices of degree~$1$. This reduction does not increase the number of vertices of degree more than~$2$ nor the number of $2$-paths between them.

We apply the following branching algorithm on an instance $(G,S)$ of {\sc Steiner Forest}. Let $P$ be a maximal $2$-path in $G$ of which both ends have degree more than~$2$. Apply the construction of Lemma~\ref{l-path} to $P$. We compute $\sft(G_1,S_1)$ and $w_1$ and for all $e,f \in E(P)$, we compute $\sft(G_2^{e,f})$ and $w_2^{e,f}$. Then we return $\min\{\sft(G_1,S_1) + w_1, \min_{e,f \in E(P)} \{\sft(G_2,S_2^{e,f}) + w_2^{e,f}\}\}$. If $G$ has no such path $P$, then $G$ is a forest of subdivided stars and we can find a minimum Steiner forest using Theorem~\ref{t-tw}.

The correctness of the above algorithm is immediate from Lemma~\ref{l-path} and Theorem~\ref{t-tw}. It remains to analyze the running time of this algorithm when called on $(G_0,S_0)$. In the branching algorithm, both in $G_1$ and in $G_2^{\cdot,\cdot}$, the number $k$ of maximal $2$-paths between vertices of degree more than~$2$ decreases by $1$.  Since the base of the branching, the time spent to determine the branches, and the number of branches in each step are all polynomial, the running time of the algorithm is $n^{O(k)}$.
\end{proof}



\section{Vertex Covers and 2-Deletions Sets}\label{s-c}
In this section, we consider {\sc Steiner Forest} on graphs with a small deletion set. Recall that the $1$-deletion number equals the vertex cover number. Here is our first result:

\begin{theorem} \label{t-vc}
{\sc Steiner Forest} is \FPT\  by the vertex cover number of the input graph.
\end{theorem}
\begin{proof}
Consider an instance of {\sc Steiner Forest} with a given vertex cover. That is, a graph $G=(V,E)$, a set $S \subseteq V \times V$ of terminal pairs, and a vertex cover $C$ of $G$. We show that we can compute a Steiner forest for $(G,S)$ of minimum size in time $2^{O(|C| \log |C|)} n^{O(1)}$. Note that if the vertex cover is not given as part of the input, then we can compute one of size equal to the vertex cover number $k$ in $2^{O(k)} n^{O(1)}$ time (see e.g.~\cite{CKX06}). Let $R := V \setminus C$.

Throughout this description, we assume two things. First, no vertex of $C$ is a terminal vertex. Indeed, if $v \in C$ is a terminal vertex, then we create a vertex $v'$ only adjacent to~$v$ and move all terminals on $v$ to $v'$. Note that $C$ is still a vertex cover of the resulting graph, the resulting instance is feasible if and only if the original is, and the size of any solution is increased by exactly~$1$. Second, we assume that the set $S$ is \emph{transitive}, i.e., if $(u,v) \in S$ and $(v,w) \in S$, then $(u,w) \in S$ for any $u,v,w \in V$. This means that the terminals can be grouped, in what we call \emph{schools}. Any school must be part of the same connected component of any Steiner forest for $(G,S)$. 

Let $h'$ be the number of schools and let $\{L_1,\ldots,L_{h'}\}$ be the set of schools. Note that $L_1,\ldots,L_{h'}$ are disjoint. Let $L_{h'+1}, \ldots, L_{h}$ denote the sets of singleton vertices in $R - \bigcup_{i=1}^{h'} L_i$; that is, $|R - \bigcup_{i=1}^{h'-1} L_i| = h-h'$. By abuse of terminology, we also call each of the sets $L_{h'+1},\ldots,L_{h}$ a school. Then $L_1,\ldots,L_{h}$ form a partition of $R$.

We use the following notation. Let $F$ be any forest in $G$. Let $I = \{I_1,\ldots,I_\ell\}$ be the set of connected components of $F - R$; in other words, this is the set of connected components induced by the edges of $F$ between vertices of $C$. Note that some of these connected components may be singletons, that not all vertices of $C$ have to appear in a connected component, and that $\ell \leq |C|$. We call $I_1,\ldots,I_{\ell}$ the {\em islands} of $F$. Note that several islands of $F$ may belong to the same connected component of $F$. For any connected component of $F$, we call the set of islands that are part of that component its {\em\arch}. 
Let $p$ be the number of \archs. We use $A_i$ to denote the set of vertices in the $i$-th \arch and let $A = \{A_1,\ldots,A_p\}$ (so $A$ is a partition of $C \setminus I_0$).

In the remainder, let $F^*$ be an arbitrary Steiner forest for $(G,S)$ of minimum size. Let $I^* = \{I^*_1,\ldots,I^*_{\ell^*}\}$ be the islands of $F^*$ and let $I^*_0$ denote the set of vertices of $C$ not in any of $I^*_1,\ldots,I^*_{\ell^*}$. Let $A^* = \{A^*_1,\ldots,A^*_{p^*}\}$ denote the set of \archs of $F^*$.

We first branch on all $2^{O(|C| \log |C|)}$ partitions of $C$. Then one of them has to be equal to $A^* \cup \{I^*_0\}$. Consider one such branch, letting $I_0$ be one part and $A = \{A_1,\ldots,A_p\}$ the set of remaining parts. That is, $A$ is a partition of $C \setminus I_0$ (this effectively adds $p+1$ further branches). We now aim to find a minimum-size Steiner forest $F$ for $(G,S)$ such that $A$ are the \archs of $F$. For any such $F$, $|E(F)| \geq |E(F^*)|$. If the pair $A, I_0$ is equal to the pair $A^*,I^*_0$, then $F^*$ is a possible solution and thus $|E(F)| = |E(F^*)|$.

The crux to computing $F$ is to determine which schools connect to which \archs and in particular, whether they help in connecting some islands to form an \arch. To this end, we perform a dynamic program. We first introduce some notation. For any two partitions $X = \{X_1,\ldots,X_p\}$ and $Y = \{Y_1,\ldots,Y_q\}$ of the same set, we use the notation $Y \refines X$ to denote that for any $i \in \{1,\ldots,q\}$, $Y_i \subseteq X_j$ for some $j \in \{1,\ldots,p\}$. Note that this implies that $q \geq p$. 
We also say that $Y_{i_1},\ldots,Y_{i_r}$ \emph{build} $X_j$ for some $i_1,\ldots,i_r \in \{1,\ldots,q\}$ and some $j \in \{1,\ldots,p\}$ if $Y_{i_1} \cup \cdots \cup Y_{i_r} = X_j$ 
(so $Y_{i_1}, \ldots, Y_{i_r}$ is a partition of $X_j$).

Recall that $I_0$ and $A = \{A_1,\ldots,A_p\}$ are fixed by the branch we are in. For any partition $B = \{B_1,\ldots,B_q\}$ of $C \setminus I_0$ such that $B \refines A$ and any integer $s \in \{0,1,\ldots,h\}$, let $T[B,s]$ be the minimum size of any forest $F'$ for $G$, such that the \archs of $F'$ are exactly $B$ and for any $L_i$ with $i \leq s$, there exists an $A_j$ for $j \in\{1,\ldots,p\}$ such that any terminal vertex in $L_i$ has a neighbour in $A_j$ in $F'$. We are looking to compute $T[A,h]$. 
After $T[B,h']$ has been computed, all terminals have already been connected to parts of $A$. However, the vertices in $L_{h'+1}, \ldots, L_h$ can still be used to connect \archs of $B$ to form $A$. Finally, if the pair $A, I_0$ is equal to the pair $A^*,I^*_0$, then $T[A,h] = |E(F^*)|$.

We start with the base case: $T[B,0]$ for some partition $B = \{B_1,\ldots,B_q\}$ of $C \setminus I_0$ such that $B \refines A$. 
Let $T[B,0]$ be equal to the sum over all $i \in \{1,\ldots,q\}$ of the size of a spanning tree of $G[B_i]$. If $G[B_i]$ has no spanning tree, we take the size to be~$\infty$. This means that $B$ is the set of islands of the (partial) solution. Note that $T[B,0] = \infty$ if $G[B_i]$ is not connected for some $i$.

Let $s \in \{1,\ldots,h\}$. We describe the subproblem optimality property. If the pair $A, I_0$ is equal to the pair $A^*,I^*_0$ and $F'_*$ is an optimal solution to the subproblem indicated by $T[B,s]$, then let $L'$ be the set of edges of $F'_*$ incident on $L_s$. Then $F'_* - L'$ is a forest such that for any $L_i$ with $i \leq s-1$, there exists an $A_j$ for $j \in\{1,\ldots,p\}$ such that any vertex in $L_i$ has a neighbour in $A_j$ in $F'_* - L'$. 
Let $B'$ denote the \archs of $F'_* - L'$. Clearly, $B' \refines B$ and by the preceding, $F'_* - L'$ is a possible solution to $T[B',s-1]$. Moreover, the edges of $L'$ connect the \archs described by $B'$ to form $B$. Then any optimal solution to the subproblem indicated by $T[B',s-1]$ union the edges of $L'$ is again a solution for $T[B,s]$. The size of this union is no larger than $F'_*$, and thus optimal.

For the algorithm, we now compute $T[B,s]$ for some partition $B = \{B_1,\ldots,B_q\}$ of $C \setminus I_0$ such that $B \refines A$.  
We consider each partition $B' = \{B'_1,\ldots,B'_{q'}\}$ of $C \setminus I_0$ such that: 

\begin{itemize}
\item $B' \refines B$; 
\item there is at most one part of $A$, say $A_1$, such that more parts of $B'$ build $A_1$ than 
parts of $B$ build $A_1$ (if this exists, call 
$A_1$ a \emph{distilled part}); and
\item for all other members $A_j$ of $A$ an equal number of parts of $B'$ build $A_j$ as do parts of $B$. 
\end{itemize}

\noindent
We use $T[B',s-1]$ to determine $T[B,s]$. Given $B'$ and $B$, we compute how to connect parts of $B'$ that build parts of $B$ using vertices of the school $L_s$ and to connect the remaining terminal vertices of $L_s$ to vertices of a single part of $A$. We now formalize this idea. 

Consider first the case that no part of $A$ is distilled by $B'$ and $B$; that is, $B'=B$. If $s \geq h'+1$, then $T[B,s] = T[B',s-1] = T[B,s-1]$, because such a singleton school $L_s$ has no terminal vertices by definition. If $s \leq h'$, then all vertices of the school should connect to vertices of a single $A_j$ for some $j \in\{1,\ldots,p\}$. For any $j$, this is either not possible (because some vertex in $L_s$ has no edges to $A_j$) or requires exactly $|L_s|$ edges. If this is not possible for every $j \in \{1,\ldots,p\}$, then it is not possible to connect $L_s$ via any single part of $A$, and we stop the dynamic program and proceed with the next branch for $A$ and $I_0$. Otherwise, let $T[B,s] = |L_s| + T[B',s-1]$.

Consider the case that (without loss of generality) $A_1$ is the part of $A$ that is distilled by $B'$ and $B$.  
Suppose that $B_1,\ldots,B_z$ build $A_1$. For any $j \in \{1,\ldots,z\}$, suppose that $B'_{j_1},\ldots,B'_{j_{z'_j}}$ build $B_j$ (note that $z'_j$ depends on $j$, but we drop this extra subscript for clarity when possible). If $z'_j > 1$, then the school $L_s$ must connect $B'_{j_1},\ldots,B'_{j_{z'}}$. 
We aim to compute a forest $F''$ such that: $E(F'') \subseteq E(G) \cap (L_s \times A_1)$; $L_s \subseteq V(F'')$; if $B'$ are the \archs of a forest $F'''$ of $G - L_s$, then $B$ are the \archs of $F''' \cup F''$; 
there is no vertex $V(F'') \cap R$ incident in $F''$ on both a vertex of $B_j$ and of $B_{j'}$ for any $j,j' \in \{1,\ldots,z\}$ with $j \not= j'$. Let $F''_*$ be such an $F''$ of minimum size.

We now aim to `guess' the pattern of how to connect parts of $B'$ that build parts of $B$ using vertices of the school $L_s$. The intuition of this pattern is as follows (see Figure~\ref{f-pattern}). Suppose that an optimal solution $F'_*$ to $T[B,s]$ is indeed formed by the union of $F''_*$ and an optimal solution to $T[B',s-1]$. Suppose we remove a vertex $v \in L_s$ from $F'_*$ such that the number of \archs of the resulting forest is larger than of $F'_*$. Then some $B_j$ that is an \arch of $F'_*$ is no longer an \arch. Indeed, the sets $B'_{j_1},\ldots,B'_{j_{z'}}$ that build $B_j$ are now partitioned into a number of parts equal to the degree of $v$ in $F'_*$; each of these parts forms an \arch of the resulting forest. If we iteratively proceed, removing vertices of $L_s$ adjacent in $F'_*$ to $B_j$ such that the number of \archs increases, then we obtain a tree-like structure that is rooted. Each internal node of this tree (except possibly the root) has degree at least~$3$ and is associated with a partition of some subset of $\{B'_{j_1},\ldots,B'_{j_{z'}}\}$; this subset is a part of the partition of its parent. The leafs of the tree are $B'_{j_1},\ldots,B'_{j_{z'}}$. We seek this tree and its associated partitions.

We now turn the above intuition into a branching process (again, refer to Figure~\ref{f-pattern}). For any $j \in \{1,\ldots,z\}$, let $T_j$ be any rooted tree on $z'_j$ labeled leafs $t_1,\ldots,t_{j_{z'}}$, where the leafs are identified with the sets $B'_{j_1},\ldots,B'_{j_{z'}}$, such that all internal vertices (except possibly the root) have degree at least~$3$. For any internal vertex $t$ of $T_j$, let $Y_j^t$ be the set of, for each child $t'$ of $t$, the subset of $\{B'_{j_1},\ldots,B'_{j_{z'}}\}$ for which the corresponding leafs are descendants of $t'$ in $T_j$ (here, $t'$ is a descendant of itself). Observe that this set is a partition of a subset of $\{B'_{j_1},\ldots,B'_{j_{z'}}\}$. Let $Y_j = \bigcup_t \{Y_j^t\}$. We branch over all possible $Y_j$ for all $j \in \{1,\ldots,z\}$. 
Note that an optimal such set $Y_j^*$ and tree $T_j^*$ may be obtained by the iterative process described above. Using Cayley's formula~\cite{B60,C89}, we note that there are $2^{O(|C| \log |C|)}$ branches.

\begin{figure}[t!]
\begin{center}
\includegraphics[scale=0.7]{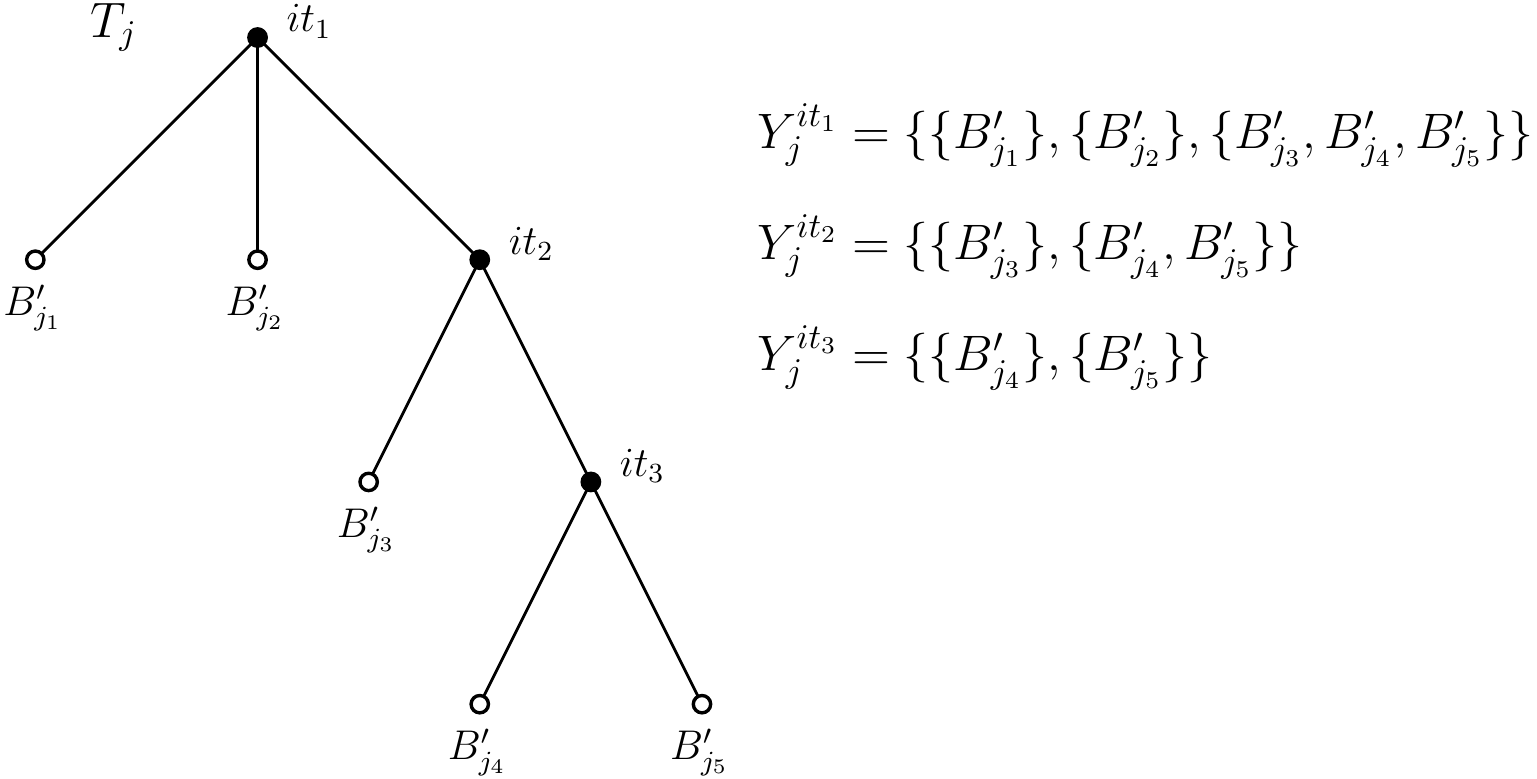}
\end{center}
\caption{A possible pattern $T_j$ where $z'_j = 5$. The tree has internal vertices $it_1, it_2, it_3$, and the associated $B'_{j}$ are listed at the leafs. The sets $Y_j^{t}$ are listed on the right, for all three of the internal tree vertices. Each internal vertex $it$ corresponds to a vertex of $L_s$ that connects the sets $B'_j$ in $Y^{it}_j$. For example, $it_2$ corresponds to a vertex of $L_s$ that has degree~$2$ in the solution: one edge to a vertex in $B'_{j_3}$ and one edge to a vertex in $B'_{j_4}$ or in $B'_{j_5}$. The perfect matching algorithm finds an optimal assignment of vertices in $L_s$ to internal tree vertices of $T_j$, such that the pattern describes the connection structure between the sets $B'_j$.}\label{f-pattern}
\end{figure}

We define some helpful variables. 
Let $\Gamma$ be the set of indices $j \in \{1,\ldots,z\}$ for which $z'_j > 1$. 
Let $\beta = \sum_{j \in \Gamma} |Y_j|$; this is the number of distinct vertices of $L_s$ that are needed to connect the parts of $B'$ that build parts of $B$ according to the pattern prescribed by the $Y_j$. If $\beta$ is larger than $|L_s|$, then proceed to the next branch, as then there are not sufficient vertices in $L_s$ to do this. 
Let $\sigma$ be a bijection between $\{1,\ldots,\beta\}$ and the $\bigcup_{j \in \Gamma} Y_j$ so that we may speak of $Y^{\sigma(i)}$ to refer to the members of the $Y_j$. That is, each $Y^{\sigma(i)}$ corresponds to $Y_j^t$ for some $j$ and some internal vertex $t$ of $T_j$; in particular, it is a partition of a subset of $\{B'_{j_1},\ldots,B'_{j_{z'}}\}$.

To create the assignment of vertices of $L_s$ to the pattern, we construct an instance of minimum-weight perfect matching on an auxiliary bipartite graph $H$. Let $\alpha = |L_s|$. Create vertices $u_1,\ldots,u_\alpha$ (we identify these with the vertices in $L_s$); these form one side of the bipartition. Add vertices $v_1,\ldots,v_{\alpha-\beta}$ and vertices $x_1,\ldots,x_\beta$; these form the other side of the bipartition. Consider any $i \in \{1,\ldots,\beta\}$. For any $u_k$ with $k \in \{1,\ldots,\alpha\}$ such that $u_k$ is adjacent in $G$ to a vertex in each of the sets $\{B'_{i_1},\ldots,B'_{i_l}\}$ that build $Y^{\sigma(i)}$, add an edge to $H$ between $u_k$ and $x_i$ of weight equal to $|Y^{\sigma(i)}|$. This represents the cost of using $u_k$ to connect the parts of the partition $Y^{\sigma(i)}$, thus helping to build the \archs. 
Finally, for each $k \in \{1,\ldots,\alpha\}$ and each $i \in \{1,\ldots,\alpha-\beta\}$, if $u_k$ is adjacent in $G$ to a vertex of $A_1$, then 
add an edge between $u_k$ and $v_i$ to $H$ of weight equal to either $1$ if $s \leq h'$ and $0$ otherwise. 
This represents the cost of directly connecting $u_k$ to $A_1$.

Now find a minimum-weight perfect matching in $H$. This corresponds to a forest $F'''$ that connects each of the parts of the $Y_j$ for each $j \in \Gamma$ in a manner of minimum size. Hence, $F'''$ is a candidate for a solution $F''$. We return the best solution we find over all branches. Since we encounter $Y_j^*$ as a candidate branch, we find a solution of size at most the size of $F''_*$. Following all the previous arguments and the description of the algorithm, we return a minimum Steiner forest for $(G,S)$.

To show that the running time of our algorithm is $2^{O(|C| \log |C|)} n^{O(1)}$, note that there are $2^{O(|C| \log |C|)}$ candidates for $I_0$ and $A$ and $2^{O(|C| \log |C|)}$ entries for the dynamic program. For each entry, we consider $2^{O(|C| \log |C|)}$ candidates for $B'$ and branch on $2^{O(|C| \log |C|)}$ candidates for $Y_j$. Finding a minimum-weight perfect matching takes polynomial time in the size of the instance~\cite{EK72}, which is $O(|L_s|^2)$. 
Finally, note that $h \leq n$.
\end{proof}

\begin{remark}
The \FPT\ algorithm extends straightforwardly to the setting with edge weights.
\end{remark}

\medskip
\noindent
We also need the next result, which shows the polynomial part of Theorem~\ref{t-2d2}.

\begin{lemma}\label{l-2d2}
{\sc Steiner Forest} is polynomial-time solvable for graphs with a $2$-deletion set of size at most~$2$.
\end{lemma}

\begin{proof}
Let $G$ be an $n$-vertex graph that, together with a set $S$ of terminal pairs, is an instance of {\sc Steiner Forest}. If $G$ has a $2$-deletion set of size $1$, then this vertex forms a cut vertex and we can apply Lemma~\ref{l-2con} to reduce to case of graphs of size~$2$, which is trivial.  So assume that $G$ has a $2$-deletion set $C$ of size~$2$, say $C=\{u,v\}$. By Lemma~\ref{l-2con} we may assume that $G$ is $2$-connected.  If $u$ and $v$ are adjacent, then the problem is trivial.  Indeed, either the edge $uv$ is part of the solution and we may contract it to form a $2$-deletion set of size $1$, or the edge is not part of the solution and we may reduce to the case when $u$ and $v$ are not adjacent.  Hence, we assume that $u$ and $v$ are not adjacent. 

As $G$ is $2$-connected, there exists a path from $u$ to $v$ in $G$. As $C$ is a $2$-deletion set, every path from $u$ to $v$ has at most two inner vertices. We check every of the $O(n^2)$ paths from $u$ to $v$ (with at most two inner vertices). For each such a path $P$, it straightforward to find a minimum solution over all solutions that contain $P$ (as we can contract the path $P$ to a single vertex). We also compute a minimum solution over all solutions that contain $u$ but not~$v$, and a minimum solution over all solutions that contain $v$ but not $u$. This takes polynomial time as well.

It remains to find a minimum solution over all solutions with two connected components, one containing $u$ and the other one $v$, and to check this minimum solutions with the minimum solutions of the other types found above. As $C$ is a $2$-deletion set, any non-terminal vertex that is not adjacent to at least one of $u$, $v$ will not be used in a minimum solution. Hence, we remove such non-terminal vertices from $G$.  Moreover, we may restrict ourselves to solutions that contain at most three non-terminal vertices. Any solution with more than three non-terminal vertices has at least as many edges as a solution in which $u$ and $v$ are in the same connected component. Hence, we branch by considering all $O(n^3)$ options for the set of non-terminal vertices used in a minimum solution. We consider each of the $O(n^3)$ branches separately and as follows.

First we remove all non-terminal vertices that we did not guessed to be in the solution. For each of the at most three guessed non-terminal vertices we do as follows.  By construction, such a guessed non-terminal vertex $z$ is adjacent to $u$ or $v$, and we contract the edge $zu$ or $zv$, respectively. Afterwards, we apply Lemma~\ref{l-2con} again such that we may assume that the resulting instance, which we denote by $(G,S)$ again, is $2$-connected. Now every vertex of $V(G)\setminus \{u,v\}$ is a terminal vertex. Hence, every vertex of $V(G)\setminus \{u,v\}$ will be in the solution we are trying to construct.

For every connected component $D$ of $G-\{u,v\}$ we do as follows. As $C$ is a $2$-deletion set, $D$ consists of exactly two vertices $x$ and $y$. If $x$ and $y$ are both adjacent to both $u$ and $v$, then we may remove the edge $xy$ for the following reason. As $x$ and $y$ are terminal vertices, we need to connect $x$ to either $u$ or $v$, and we also need to connect $y$ to either $u$ or $v$. For doing this, we do not need to use the edge $xy$.

Now, suppose that one of $x,y$, say $x$, is adjacent to $u$ and $v$, whereas the other one, $y$, is adjacent to only one of $u$ and $v$, say to $u$. Again, as $x$ and $y$ are terminal vertices, we need to connect $x$ to either $u$ or $v$ and we also need to connect $y$ to either $u$ or $v$. If we use $xu$, then we must also use $yu$, and in that case we can replace $xu$ by $xy$. Hence, we may remove $xu$.

So, afterwards, we reduced the instance in polynomial time to a new instance, which we will also denote by $(G,S)$, with the following properties. Every connected component of $G-\{u,v\}$ has at most two vertices. By $2$-connectivity, for every connected component that contains exactly one vertex $z$ we have the edges $uz$ and $vz$. Moreover, for every connected component that contains exactly two vertices $x$ and $y$, we have the edges $ux$ and $vy$ but not $uy$ and not $vx$. 

As every vertex in $V(G)\setminus \{u,v\}$ is a terminal vertex, we need one edge in the solution for each singleton connected component of $G-\{u,v\}$ and two edges for each two-vertex connected component. First, suppose that one of $u$, $v$, say $u$, is not a terminal vertex. Then the solution that contains all vertices of $G-\{v\}$ is a minimum solution, so we can stop (note that we already found this solution before).

Now suppose that both $u$ and $v$ are terminal vertices. We discard the branch if $u$ and $v$ represent terminals of the same pair. Suppose this is not the case. If $u$ represents $s_i$, then we connect the terminal vertex that represents $t_i$ to $u$. This can only be done in one way: if the terminal vertex representing $t_i$ is not adjacent to $u$, it contains a unique neighbour adjacent to $u$. Afterwards, we remove the terminal pairs that we connected in this way from $G$. If it was not possible to connect some terminal pair, then we discard the branch. Otherwise, we apply Lemma~\ref{l-2con} again, such that we may assume that the resulting instance, which we denote by $(G,S)$ again, is $2$-connected. The other properties are maintained, and neither $u$ nor $v$ is a terminal vertex. Hence, we can take as solution the forest that has one connected component with vertex set $\{u\}$ and one connected component with vertex set $V(G)\setminus \{u\}$. 
\end{proof}

\noindent
For some of our results we need the following extension of Lemma~\ref{l-2d2}.

\begin{lemma}\label{l-extension}
{\sc Steiner Forest} is polynomial-time solvable for graphs $G$ with a set $X \subseteq E(G)$ of bounded size such that $G-X$ has a $2$-deletion set $C$ of size at most~$2$ and each end-point of every edge of $X$ is either an isolated vertex of $G-X$ or a vertex of $C$.
\end{lemma}

\begin{proof}
Let $(G,S)$ be an instance of {\sc Steiner Forest}. Let $X$ be a set that satisfies the conditions of the lemma. We apply the following branching algorithm recursively on the edges of $X$.  Let $x\in X$. We use $P$ to denote the $2$-path consisting of $x$. Apply the construction of Lemma~\ref{l-path} to $P$. We compute $\sft(G_1,S_1)$ and $w_1$ and for all $e,f \in E(P)$, we compute $\sft(G_2^{e,f})$ and $w_2^{e,f}$. We return $\min\{\sft(G_1,S_1) + w_1, \min_{e,f \in E(P)} \{\sft(G_2,S_2^{e,f}) + w_2^{e,f}\}\}$. In each branch, we have either contracted or deleted $x$. If every edge of $X$ has been branched on, then $C$ is still a $2$-deletion set of size~$2$ in $G$. Hence, we can apply Lemma~\ref{l-2d2} to $(G,S)$ respectively.

The correctness of the algorithm is immediate from Lemma~\ref{l-path} and Lemma~\ref{l-2d2}. For the running time, note that we branch $|X|$ times in a constant number of options. Since Lemma~\ref{l-2d2} gives an algorithm with polynomial running time, the entire algorithm runs in polynomial time.
\end{proof}

\section{Polynomial Cases}\label{s-poly}

We start with a general lemma, using Theorem~\ref{t-vc}.

\begin{lemma}\label{l-sp2}
For a graph $H$, if {\sc Steiner Forest} can be solved in polynomial time on the class of $H$-subgraph-free graphs, then {\sc Steiner Forest} can be solved in polynomial time on the class of $H+P_2$-subgraph-free graphs.
\end{lemma}

\begin{proof}
Let $(G,S)$ be an instance of {\sc Steiner Forest} such that $G$ is $(H+P_2)$-subgraph-free. If $G$ is $H$-subgraph-free, then {\sc Steiner Forest} can be solved in polynomial time by assumption. Hence, we may assume that $G$ has a subgraph $H'$ isomorphic to $H$. Then the connected components of $G-V(H')$ are $P_2$-subgraph-free and thus have size~$1$. Hence, $H$ is a $1$-deletion set and we apply Theorem~\ref{t-vc}.
\end{proof}

\noindent
We now consider a number of graphs~$H$. The first one is the case $H=2K_{1,3}$.

\begin{lemma}\label{l-2k13}
{\sc Steiner Forest} is polynomial-time solvable for $2K_{1,3}$-subgraph-free graphs.
\end{lemma}

\begin{proof}
Let $(G,S)$ be an instance of {\sc Steiner Forest} where $G$ is $2K_{1,3}$-subgraph-free.
First, suppose that $G$ is $K_{1,3}$-subgraph-free. Then, $G$ has no vertices of degree~$3$ or more. Hence, $G$ is a disjoint union of cycles and paths, and thus $G$ has treewidth at most~$2$. We apply Theorem~\ref{t-tw}.
Now suppose that $G$ is not $K_{1,3}$-subgraph-free. We may assume that $G$ is $2$-connected by Lemma~\ref{l-2con}. Let $v$ be a vertex of maximum degree. Since $G$ is not $K_{1,3}$-subgraph-free, $v$ has degree at least~$3$. Let $v_1,v_2,v_3$ be any three distinct neighbours of $v$. We call $A = \{v,v_1,v_2,v_3\}$ an \emph{antares}.

We prove a sequence of claims about $2K_{1,3}$-subgraph-free graphs and antares.\\[-15pt]

\begin{claim}\label{p-2k13-degree}
If $v$ has degree more than $6$, then every other vertex of $G$ has degree at most~$6$.
\end{claim}

\begin{claimproof}
For the sake of contradiction, let $u,v \in V(G)$ both be vertices of degree more than $6$. Pick three neighbours $u_1,u_2,u_3$ of $u$ and three neighbours $v_1,v_2,v_3$ of $v$ such that $u_1,u_2,u_3,v_1,v_2,v_3$ are distinct and not equal to $u$ or $v$. Because $|(N(u) \setminus\{v\}) \cup (N(v)\setminus\{u\})| \geq 6$ and $|N(u)|,|N(v)| \geq 7$, this is always possible. However, $u,u_1,u_2,u_3$ and $v,v_1,v_2,v_3$ form a subgraph isomorphic to $2K_{1,3}$ in $G$, a contradiction.
\end{claimproof}
\vspace*{-20pt}

\begin{claim}\label{p-2k13-ccs}
The graph $G-A$ has at most~$15$ connected components.
\end{claim}

\begin{claimproof}
Since $G$ is $2$-connected, any component of $G-A$ has an edge to at least one of $v_1,v_2,v_3$. By Claim~\ref{p-2k13-degree} and the fact that $v$ has maximum degree in $G$, we find that $v_1,v_2,v_3$ each have degree at most~$6$. Recall that each of $v_1,v_2,v_3$ has an edge to $v$. Hence, $G-A$ has at most~$15$ connected components.
\end{claimproof}
\vspace*{-10pt}

\begin{claim} \label{p-2k13-paths}
If $v$ has degree more than~$6$, then $G-A$ is a union of paths. Moreover, only the ends of such a path can be adjacent to $v_1$, $v_2$, or $v_3$.
\end{claim}

\begin{claimproof}
As $G-A$ is $K_{1,3}$-subgraph-free, $G-A$ is a disjoint union of cycles and paths. Suppose that $G-A$ has a connected component $D$ that is a cycle. Since $G$ is $2$-connected, $D$ is adjacent to at least one of $v_1,v_2,v_3$, say $v_1$. Let $x$ be a vertex of $D$ adjacent to $v_1$. By assumption and the definition of an antares, $v$ has degree at least $7$. Hence, it has three neighbours that are not $v_1$, $x$, or the two neighbours of $x$ on $D$. Hence, $G$ has a $2K_{1,3}$ as a subgraph, a contradiction. The same argument holds if $D$ is a path and $x$ is an internal vertex of this path adjacent to $v_1$, $v_2$, or $v_3$.
\end{claimproof}

\noindent
We continue as follows. Firstm suppose $v$ (and thus any other vertex of $G$) has degree at most~$6$.
Then $G-A$ is $K_{1,3}$-subgraph-free, and all vertices of $G-A$ have degree at most~$2$. Hence, all vertices of degree more than~$2$ in~$G$ are in $A$ or have an edge to a vertex of $A$. Hence, using the same argument as in Claim~\ref{p-2k13-ccs}, there are at most $4 + (3 \cdot 5 + 3) = 22$ vertices of degree more than~$2$ in~$G$. Each of them has degree at most~$6$, so there is a constant number of maximal $2$-paths between them. We apply Lemma~\ref{l-degree}.

Now assume that $v$ has degree more than~$6$. Let $A=\{v,v_1,v_2,v_3\}$ be an antares in $G$.  We now apply a single but complex branching step. After this step, we will argue that all $2$-connected components of the remaining graphs have treewidth at most~$2$ or have maximum degree at most~$33$. Note that in a $2$-connected, $2K_{1,3}$-subgraph-free graph of maximum degree~$33$, it follows via an argument as in Claim~\ref{p-2k13-ccs} and a calculation as above that it has at most~$49$ vertices of degree more than~$2$. Each of them, except $v$, has degree at most~$6$ by Claim~\ref{p-2k13-degree}, so there is a constant number of maximal $2$-paths between them. Then, by combining Theorem~\ref{t-tw} and Lemma~\ref{l-degree} with Lemma~\ref{l-2con}, we can solve each branch in polynomial time.

For the branching, we first observe that $G-A$ is a union of paths, $P_1,\ldots,P_a$, by Claim~\ref{p-2k13-paths}. By Claim~\ref{p-2k13-ccs}, $a \leq 15$. Moreover, by Claim~\ref{p-2k13-paths}, only the ends $u_{P_i},w_{P_i}$ of each such path $P_i$ can be adjacent to $v_1$, $v_2$, or $v_3$. In any optimal solution $F$, the intersection of $F$ and $P_i$ consists of multiple subpaths $P^1_i,\ldots,P^{b_i}_i$ of $P_i$. We assume that these subpaths are numbered as they occur along $P_i$, where $P^1_i$ has its vertices closest to $u_{P_i}$. Note that possibly $b_i = 0$ and that possibly some subpaths consist of a single vertex. Observe that only $P^1_i$ and $P^{b_i}_i$ can be adjacent to $v_1,v_2,v_3$, if $u_{P_i} \in V(P^1_i)$ and $w_{P_i} \in V(P^{b_i}_i)$ respectively. Moreover, there can be at most one edge between $P^j_i$ and $v$ in $F$. Denote these edges by $e^1_i,\ldots,e^{b_i}_i$, where $e^j_i$ is the edge between $P^j_i$ and $v$ in $F$. Note that if some $P^j_i$ is a connected component of the solution by itself, then $e^j_i$ does not exist.

We now branch on what $P^1_i$, $P^{b_i}_i$, $e^1_i$, and $e^{b_i}_i$ are (if they exist). Note that $P^1_i$ and $P^{b_i}_i$ can be described by their ends; since $a \leq 15$, the number of ends is at most~$60$ in total, which we choose among $O(n)$ vertices. Similarly, we need to choose at most~$30$ edges of the form $e^1_i$ and $e^{b_i}_i$ among $O(n)$ edges incident on $v$. Hence, the number of branches\footnote{This can be reduced to $n^{60}$ by a  more complex argument, but this is beyond the focus of our paper.} is bounded by $O(n^{90})$. We explicitly allow the paths $P^1_i$ and $P^{b_i}_i$ to be empty (undefined) or to consist of a single vertex, and that the edges $e^1_i$ and $e^{b_i}_i$ are undefined.

Let $Q^1_i$, $Q^{b_i}_i$, $f^1_i$, and $f^{b_i}_i$ be the guessed paths and edges for all $P_i$ in a branch~$B$. Note that we do not guess the precise value of $b_i$, but use it as a placeholder here; we can have a branch in which $b_i = 0$ and $Q^1_i$, $Q^{b_i}_i$, $f^1_i$, and $f^{b_i}_i$ are undefined. 
We also discard a branch if $f^1_i$ is defined but not incident on a vertex of $Q^1_i$ and similarly if $f^{b_i}_i$ is defined but not incident on a vertex of $Q^{b_i}_i$. 
We apply substantial abuse of notation in the below, also if say $Q^1_i = Q^{b_i}_i$ or $Q^1_i = P_i$. However, this does not harm the overall argument.

For each $P_i$, remove all edges incident both on $v$ and on any vertex between $u_{P_i}$ and the end of $Q^1_i$ that is furthest from $u_{P_i}$, except the edge $f^1_i$. Similarly, remove all edges incident both on $v$ and on any vertex between $w_{P_i}$ and the end of $Q^{b_i}_i$ that is furthest from $w_{P_i}$, except the edge $f^{b_i}_i$. Note that possibly $b_i = 0$, in which case all edges from $v$ to $P_i$ are removed. This is correct, because in any Steiner forest for $(G,S)$ that contains $Q^1_i$ ($Q^{b_i}_i$) and $f^1_i$ ($f^{b_i}_i$) cannot also contain any of the edges that were removed, or either the forest would contain a cycle or $Q^1_i$ ($Q^{b_i}_i$) is not the first (last) path of solution restricted to $P_i$. If $b_i = 0$, then any Steiner forest for $(G,S)$ that does not contain any vertices of $P_i$ cannot contain edges from $v$ to $P_i$. In particular, this holds for the considered (but unknown) optimum solution $F$. Let $G_B$ be the resulting graph for this branch. Note that $G_B$ is a subgraph of $G$.

We claim that each $2$-connected component of $G_B$ has treewidth at most~$2$ or maximum degree~$33$. Let $Q^m_i$ be the maximal subpath of $P_i$ that is between and disjoint from $Q^1_i$ and $Q^{b_i}_i$. Note that $Q^m_i$ does not exist if $b_i = 0$ or $V(Q^1_i) \cup V(Q^{b_i}_i) = V(P_i)$, but this does not harm the overall argument. Note that $Q^m_i$ is only adjacent to $v$ in $G_B$ (and not to $v_1$, $v_2$, or $v_3$). Hence, they form a $2$-connected component of $G_B$ of treewidth at most~$2$. Any other $2$-connected component of $G_B$ contains a selection of $v,v_1,v_2,v_3$ and the vertices of the $Q^1_i$'s and $Q^{b_i}_i$'s. However, each $Q^1_i$ and each $Q^{b_i}_i$ has at most one edge to $v$, and there are at most~$30$ of them. Hence, $v$ has degree at most~$33$ in each such $2$-connected component. The claim follows.

Note that if a $2$-connected component has treewidth at most~$2$, the corresponding instance of {\sc Steiner Forest} can be solved in polynomial time by Theorem~\ref{t-tw}. If a $2$-connected component has maximum degree~$33$, then using the same argument as in Claim~\ref{p-2k13-ccs}, there are at most $4 + (3 \cdot 5 + 30) = 49$ vertices of degree more than~$2$ in $G_B$. Each of them, except $v$, has degree at most~$6$ by Claim~\ref{p-2k13-degree}, so there is a constant number of maximal $2$-paths between them. Then the corresponding instance of {\sc Steiner Forest} can be solved in polynomial time by Lemma~\ref{l-degree}. Note that we can always reduce it to $2$-connected components by using Lemma~\ref{l-2con}.

Now it suffices to observe that there are $O(n^{90})$ branches and as argued above, in each branch $B$, the instance $(G_B,S)$ can be solved in polynomial time. Hence, the entire algorithm runs in polynomial time. Since one of these branches corresponds to $P^1_i$, $P^{b_i}_i$, $e^1_i$, and $e^{b_i}_i$, the algorithm is correct.
\end{proof}

\noindent
By using Lemma~\ref{l-extension} (which generalized Lemma~\ref{l-2d2}) we can extend Lemma~\ref{p-2k13-degree} as follows. 

\begin{lemma}\label{l-2k13+p3}
{\sc Steiner Forest} is polynomial-time solvable for $(2K_{1,3}+P_3)$-subgraph-free graphs.
\end{lemma}

\begin{proof}
Let $G$ be a $(2K_{1,3}+P_3)$-subgraph-free graph. If $G$ is $2K_{1,3}$-subgraph-free, we use Lemma~\ref{l-2k13}. Hence, we may assume $G$ contains a $2K_{1,3}$. Let $H$ be a subgraph of $G$ isomorphic to $2K_{1,3}$. 

Now every connected component of $G - H$ has size at most $2$, as it is $P_3$-subgraph-free. If more than $12$ such components are adjacent to the leafs of the $2K_{1,3}$'s, there must be a leaf of the $2K_{1,3}$ adjacent to at least $3$ components. But then this leaf with those neighbours form a $K_{1,3}$, and the rest of its original $K_{1,3}$ in $H$ is a $P_3$, and so the graph contains a $2K_{1,3}+P_3$. We conclude that at most~$12$ connected components of $G - H$ are adjacent to the leafs of the $2K_{1,3}$. As each of these components has size at most $2$ and there are a constant number of them, the leafs of the $2K_{1,3}$ together with all components adjacent to them is in total a subgraph $H'$ of constant size.

Observe that $H'$ is only adjacent to the roots of the $2K_{1,3}$'s that make up $H$. Hence, a constant number of edges is incident on vertices of $H'$. The other components of $G-H$, each of size at most~$2$, are also only adjacent to these roots. Hence, the roots form a $2$-deletion set of $G-E(H')$. Then, it follows that {\sc Steiner Forest} can be solved in polynomial time by Lemma~\ref{l-extension}.
\end{proof}

\noindent
We now prove the case $H=S_{1,1,4}$.

\begin{lemma}\label{l-s114}
{\sc Steiner Forest} is polynomial-time solvable for $S_{1,1,4}$-subgraph-free graphs.
\end{lemma}

\begin{proof}
	We prove that {\sc Steiner Forest} is polynomial-time solvable for $S_{1,1,r}$-subgraph-free graphs for increasing $r$, $1\leq r \leq 4$.
	In this proof, let $(G,S)$ be an instance of {\sc Steiner Forest} where $G$ is a $S_{1,1,r}$-subgraph-free graph. By Lemma~\ref{l-2con} we may assume that $G$ is $2$-connected.
	
	\medskip
	\noindent
	{\bf Case 1.} $r=1$.\\
	As $S_{1,1,1} = K_{1,3}$, we find that $G$ is $K_{1,3}$-subgraph-free graph, so we can use Lemma~\ref{l-2k13}.
	
	\medskip
	\noindent
	{\bf Case 2.} $r=2$.\\
	Assume $G$ contains $S_{1,1,1}$ as a subgraph, else apply Case 1. Let $s$ be the degree-$3$ vertex of the $S_{1,1,1}$ and let $a,b,c$ be its degree-1 neighbours (these vertices may have higher degree in $G$). Any other vertex in $G$ cannot be adjacent to any of $a,b,c$, as $G$ is $S_{1,1,2}$-subgraph-free. If $G$ has more than $4$ vertices, then $s$ is a cut-vertex, which contradicts the $2$-connectivity of $G$. Hence, the instance is trivial.
	
	\medskip
	\noindent
	{\bf Case 3.} $r=3$.\\
	Assume $G$ contains $S_{1,1,2}$ as a subgraph, else apply Case 2. Let $s$ be the degree-$3$ vertex of the $S_{1,1,2}$ and let $a_1,b,c$ be its neighbours and $a_2$ the neighbour of $a_1$ (these vertices may have higher degree in $G$). Consider the connected components of $G - \{s,a_1,b,c,a_2\}$. Let $D$ be such a component containing an edge. Note that $D$ cannot be adjacent to $a_2$ in $G$, as $G$ is $S_{1,1,3}$-subgraph-free. Moreover, $D$ cannot be adjacent to $a_1,b,c$ in $G$, as $G$ is $S_{1,1,3}$-subgraph-free and $D$ contains an edge. But then $s$ is a cut-vertex in $G$, which contradicts the $2$-connectivity of $G$. Hence, no such component $D$ with an edge exists, and all connected components of $G - \{s,a_1,b,c,a_2\}$ must be single vertices. But then $G$ has a vertex cover of size $5$, and we can solve {\sc Steiner Forest} in polynomial time by Theorem~\ref{t-vc}.
	
	\medskip
	\noindent
	{\bf Case 4.} $r=4$.\\
	Assume $G$ contains $S_{1,1,3}$ as a subgraph, else apply Case 3.
	Let $s$ be the degree-$3$ vertex of the $S_{1,1,3}$, and let $a_1,b,c$ be its neighbours and $a_2$ the neighbour of $a_1$, and $a_3$ the neighbour of $a_2$ (these vertices may have higher degree in $G$). Consider the connected components of $G' = G - \{s,a_1,b,c,a_2,a_3\}$. 
	We prove the following claims.
	
	\begin{claim} \label{claim_s114_1}
		Any connected component of $G'$ of size at least~$2$ is not adjacent to $a_2,a_3$.
	\end{claim}
	\begin{claimproof}
		A connected component of size at least~$2$ contains an edge. Hence, if it is adjacent to either $a_2$ or $a_3$, there is a path of length $4$ from $s$ to the component, forming an $S_{1,1,4}$, a contradiction.
	\end{claimproof}
	
	\begin{claim} \label{claim_s114_2}
		Both $b$ and $c$ are adjacent to at most one connected component of $G'$.
	\end{claim}
	\begin{claimproof}
		If say $b$ is adjacent to more than one connected component of $G'$, then $b s a_1 a_2 a_3$ is a length-$4$ path and forms a $S_{1,1,4}$ together with the adjacencies of the connected components, a contradiction.
	\end{claimproof}
	
	\begin{claim} \label{claim_s114_3}
		There are no two connected components $D_1,D_2$ of $G'$ such that both are adjacent to $a_1$ and $D_1$ is adjacent to $b$ or $c$.
	\end{claim}
	\begin{claimproof}
		Suppose there are two such components $D_1,D_2$. Let $d_1\in V(D_1)$ and $d_2\in V(D_2)$ be neighbors of $a_1$. If $d_1b \in E(G)$, then $a_1 d_1 b s c$ is a length-$4$ path and $a_1$ has two other neighbours $d_2, a_2$, which together form a $S_{1,1,4}$, contradicting that $G$ is $S_{1,1,4}$-subgraph-free. The argument straightforwardly extends to the case that $b$ is adjacent to any other vertex of $D_1$.
	\end{claimproof}
	
	\begin{claim} \label{claim_s114_4}
		There are no two connected components $D_1,D_2$ of $G'$, with $|V(D_1)|\geq 2$, such that $a_1$ is adjacent to both $D_1$ and $D_2$.
	\end{claim}
	\begin{claimproof}
		If there are two such components $D_1,D_2$, then by Claims~\ref{claim_s114_1},\ref{claim_s114_3}, $D_1$ can only have edges to $s$ other than having edges to $a_1$. By $2$-connectivity of $G$, $D_1$ has an edge to $s$. If $D_1$ contains two vertices $d_1^1, d_1^2$ such that $(d_1^1, a_1), (d_1^2, s) \in E$, then $G$ contains an $S_{1,1,4}$: $a_1 d_1^1\ldots d_1^2 s b$ is at least a length-$4$ path, and $a_1$ has two other neighbours $a_2, d_2$. If $D_1$ contains no two such vertices, it contains a cut-vertex: there must be two adjacencies from $D_1$ to $a_1,s$, but they cannot originate from different vertices of $D_1$. Hence, there cannot be two such connected components $D_1,D_2$ adjacent to $a_1$.
	\end{claimproof}
	
	\begin{claim} \label{claim_s114_5}
		There are at most three connected components of $G'$ of size at least~$2$: at most one adjacent to $a_1$, at most one adjacent to $b$, and at most one adjacent to $c$.
	\end{claim}
	\begin{claimproof}
		By Claim~\ref{claim_s114_4}, $a_1$ is adjacent to at most one component of size at least~$2$. Hence, by Claim~\ref{claim_s114_1}, all other components of size at least~$2$ are not adjacent to $a_1,a_2,a_3$. By $2$-connectivity, any other component of $G'$ is adjacent to at least two vertices of $s,b,c$. By the pigeonhole principle, if there are at least three other components of size at least~$2$, at least two are adjacent to either $b$ or $c$, contradicting Claim~\ref{claim_s114_2}.
	\end{claimproof}
	
	\begin{claim} \label{claim_s114_6}
		Any connected component of $G'$ has size at most~$3$.
	\end{claim}
	\begin{claimproof}
		Assume $D$ is a component of $G'$ of size at least~$4$. By Claim~\ref{claim_s114_1}, $D$ is not adjacent to both $a_2$ and $a_3$. By $2$-connectivity, $D$ is adjacent to at least one of $a_1,b,c$. Let $d$ be a vertex of $D$ adjacent to one of $a_1,b,c$. Then $d$ cannot be the endpoint of a $P_3$ in $D$, as otherwise this $P_3$ with the adjacency among $a_1,b,c$ makes a length-4 path to $s$, and hence a $S_{1,1,4}$. We get that $D$ is a star with center $d$ and no other vertex in $V(D) \setminus \{d\}$ is adjacent to $a_1, b,$ or $c$ (any such vertex is an endpoint of a $P_3$ in $D$). By 2-connectivity, another vertex $d' \neq d$ in $D$ is adjacent to $s$ or one of $a_1,b,c$, different from a neighbour of $d$. But then there is a path from $d$ to $a_2$ through $d'$ of length at least four. As $D$ is a component of at least four vertices, $d$ has two neighbours other than $d'$ in $D$, and hence, a $S_{1,1,4}$ is formed. Hence, in $G'$, no connected component exists of size at least~$4$.
	\end{claimproof}
	
	\noindent
	By Claims~\ref{claim_s114_5} and \ref{claim_s114_6} we know that in $G'$ there are at most three components of size at least~$2$, and they have size at most~$3$. These components are adjacent only to $s,a_1,b,c$. All other components are adjacent only to $s,a_1,b,c,a_2$ and have size at most~$1$. It follows that $s,a_1,b,c,a_2$ together with the vertices of all components of size at least~$2$ form a vertex cover of constant size of the graph. 
	We apply Theorem~\ref{t-vc}.
\end{proof}

\noindent
We use Lemma~\ref{l-s114} to prove the case $H=P_9$. We also need Lemma~\ref{l-2d2} again.

\begin{lemma}\label{l-p9}
{\sc Steiner Forest} is polynomial-time solvable for $P_9$-subgraph-free graphs.
\end{lemma}

\begin{proof}
	Let $G=(V,E)$ be a $P_9$-subgraph-free graph that is part of an instance of {\sc Steiner Forest}. By Lemma~\ref{l-2con} we may assume that $G$ is $2$-connected.
	Let $P=u_1\cdots u_r$, for some $r\geq 2$, be a longest (not necessarily induced) path in $G$.  As $G$ is $P_9$-subgraph-free, we find that $r\leq 8$.
	If $r\leq 5$, then $G$ is $P_6$-subgraph-free, and thus $S_{1,1,4}$-subgraph-free, and we can apply Lemma~\ref{l-s114}. Hence, $r\in \{6,7,8\}$.
	
	\medskip
	\noindent
	{\bf Case 1.} $r=6$.\\
	Then $G$ is $P_7$-subgraph-free. Suppose $G-V(P)$ has a connected component~$D$ with more than one vertex. As $G$ is $P_7$-subgraph-free and $|V(D)|\geq 2$, no vertex of $D$ is adjacent to $u_1$, $u_2$, $u_5$ or $u_6$. As $G$ is connected, at least one of $u_3$ or $u_4$, say $u_3$, has a neighbour $v$ in $D$. Suppose $w\in V(D)$ is adjacent to $u_4$ (where $w=v$ is possible). Let $vQw$ be a path from $v$ to $w$ in $D$; note that $Q$ might be empty. Now the path $u_1u_2u_3vQwu_4u_5u_6$ has at least seven vertices, contradicting that $G$ is $P_7$-subgraph-free. Hence, no vertex of $D$ is adjacent to $u_4$. This means that $u_3$ is a cut-vertex of $G$, contradicting the $2$-connectivity. We conclude that every connected component of $G-V(P)$ consists of one vertex. In other words, $\{u_1,\ldots,u_6\}$ is a vertex cover of $G$, and we can apply Theorem~\ref{t-vc}.
	
	\medskip
	\noindent
	{\bf Case 2.} $r=7$.\\
	Then $G$ is $P_8$-subgraph-free. Suppose $G-V(P)$ has a connected component~$D$ with more than one vertex. As $G$ is $P_8$-subgraph-free and $|V(D)|\geq 2$, no vertex of $D$ is adjacent to $u_1$, $u_2$, $u_6$ or $u_7$. As $G$ is connected, at least one of $u_3$, $u_4$ or $u_5$ has a neighbour $v$ in $D$. 
	
	First assume that $u_3$ or $u_5$, say $u_3$, has a neighbour $v$ in $D$. As $|V(D)|\geq 2$, we find that $v$ has a neighbour~$w$ in $D$. If $w$ has a neighbour~$x\neq v$ in $D$, then $xwvu_3u_4u_5u_6u_7$ is a path on eight vertices, contradicting that $G$ is $P_8$-subgraph-free. Hence, $v$ is the only neighbour of $w$ in $D$. 
	As $G$ is $2$-connected, this means that $w$ has a neighbour on $P$. Recall that 
	no vertex of $D$ is not adjacent to $u_1$, $u_2$, $u_6$ or $u_7$. If $w$ is adjacent to $u_4$, then $u_1u_2u_3vwu_4u_5u_6$ is a path on eight vertices. If $w$ is adjacent to $u_5$, then $u_1u_2u_3vwu_5u_6u_7$ is a path on eight vertices. Hence, $w$ must be adjacent to $u_3$ (and $u_3$ is the only neighbour of $w$ on $P$). We now find that $w$ is the only neighbour of $v$ in $D$, as otherwise, if $v$ has a neighbour $w'\neq w$ on $D$, then $w'vwu_3u_4u_5u_6u_7$ is a path on eight vertices. In other words, $V(D)=\{v,w\}$. By the same arguments, but now applied on $v$, we find that $u_3$ is the only neighbour of $v$ on $P$. Hence, $u_3$ is a cut-vertex of $G$, contradicting the $2$-connectivity of $G$.
	
	From the above we conclude that no vertex of $D$ is adjacent to $u_3$. By the same reason, no vertex of $D$ is adjacent to $u_5$. We find that $u_4$ disconnects $D$ from the rest of $G$, contradicting the $2$-connectivity of $G$. We conclude that every connected component of $G-V(P)$ consists of one vertex. In other words, $\{u_1,\ldots,u_7\}$ is a vertex cover of $G$, and we can apply Theorem~\ref{t-vc}.
	
	\medskip
	\noindent
	{\bf Case 3.} $r=8$.\\
	Recall that $G$ is $P_9$-subgraph-free, so this is the last case to consider. If every connected component of $G-V(P)$ consists of one vertex, then $\{u_1,\ldots,u_8\}$ is a vertex cover of $G$, and we can apply Theorem~\ref{t-vc}. 
	
	Now suppose that $G-V(P)$ has a connected component~$D$ with more than one vertex. As $G$ is $P_9$-subgraph-free and $|V(D)|\geq 2$, no vertex of $D$ is not adjacent to $u_1$, $u_2$, $u_7$ or $u_8$. As $G$ is connected, at least one of $u_3$, $u_4$, $u_5$ or $u_6$ has a neighbour in $D$. 
	
	First, suppose that neither $u_3$ nor $u_6$ has a neighbour in $D$. Then $u_4$ or $u_5$, say $u_4$, has a neighbour $v$ in $D$. As $G$ is $2$-connected, there exists a path $vQu_5$ from $v$ to $u_5$ that does not contain $u_4$. As no vertex from $\{u_1,u_2,u_3,u_6,u_7,u_8\}$ has a neighbour in $D$, the vertices of $Q$ belong to $D$. Now,
	$u_1u_2u_3u_4vQu_5u_6u_7u_8$ is a path on nine vertices, contradicting that $G$ is $P_9$-subgraph-free.
	
	Hence, at least one of $u_3$ or $u_6$, say $u_3$, has a neighbour $v$ in $D$. As $|V(D)|\geq 2$, we find that $v$ has a neighbour~$w$ in $D$. 
	If $w$ has a neighbour~$x\neq v$ in $D$, then $xwvu_3u_4u_5u_6u_7u_8$ is a path on nine vertices, contradicting that $G$ is $P_9$-subgraph-free. Hence, $v$ is the only neighbour of $w$ in $D$. 
	As $G$ is $2$-connected, this means that $w$ has a neighbour on $P$. Recall that 
	no vertex of $D$ is adjacent to $u_1$, $u_2$, $u_7$ or $u_8$. If $w$ is adjacent to $u_4$, then $u_1u_2u_3vwu_4u_5u_6u_7$ is a path on nine vertices. If $w$ is adjacent to $u_5$, then $u_1u_2u_3vwu_5u_6u_7u_8$ is a path on nine vertices. Hence, $w$ must be adjacent to either or both $u_3$ and $u_6$ (and $w$ has no other neighbours on $P$).
	We now find that $w$ is the only neighbour of $v$ in $D$ for the following reason. Suppose $v$ has a neighbour $w'\neq w$ on $D$. If $w$ is adjacent to $u_3$, then $w'vwu_3u_4u_5u_6u_7u_8$ is a path on nine vertices. If $w$ is adjacent to $u_6$, then $w'vwu_6u_5u_4u_3u_2u_1$ is a path on nine vertices.
	In other words, $V(D)=\{v,w\}$. By the same arguments, but now applied on $v$, we find that apart from $u_3$, it holds that $v$ may have only $u_6$ as a neighbour of $P$.
	
	Note that we have proven for any path $Q=q_1\dots q_8$ on eight vertices that every connected component of size at least~$2$ in $G-Q$ has size exactly~$2$, and moreover, $q_3$ and $q_6$ are the only vertices of $Q$ with neighbours in such a connected component. We may assume without loss of generality that $v$ is adjacent to $u_3$ and $w$ is adjacent to $u_6$, as otherwise one of $u_3,u_6,v,w$ is a cut-vertex of $G$, contradicting the $2$-connectivity of $G$.
	By replacing $P$ with $P'=u_1u_2u_3vwu_6u_7u_8$, we find that $u_4$ and $u_5$ have no neighbours outside $\{u_3,u_6\}$. 
	By replacing $P$ with $P'=u_5u_4u_3vwu_6u_7u_8$, we find that $u_2$ has no neighbours outside $\{u_3,u_6\}$ (just like $u_1$). By symmetry, $u_7$ has no neighbours outside $\{u_3,u_6\}$ (just like $u_8$). Hence, every connected component of $G-\{u_3,u_6\}$ has at most two vertices. In other words, $\{u_3,u_6\}$ is a $2$-deletion set of size at most~$2$, and we can apply Lemma~\ref{l-2d2}. 
\end{proof}

\noindent
We  use Lemma~\ref{l-p9} to show the case $H=2P_4+P_3$.

\begin{lemma}\label{l-2p4+p3}
{\sc Steiner Forest} is polynomial-time solvable for $(2P_4+P_3)$-subgraph-free graphs.
\end{lemma}

\begin{proof}
Let $G$ be a $(2P_4+P_3)$-subgraph-free graph. By Lemma~\ref{l-2con} we may assume that $G$ is $2$-connected. If $G$ is $P_9$-subgraph-free, we can solve {\sc Steiner Forest} in polynomial time by Lemma~\ref{l-p9}. Hence, we may assume $G$ contains a $P_9$, say $P = (u_1,\ldots,u_9)$, as a subgraph. As a $P_9$ contains a $2P_4$, every connected component in $G - P$ is $P_3$-subgraph-free and hence of size at most $2$. 

Let $D = (d_1,d_2)$ be a component of size exactly $2$ in $G - P$. Note that $D$ cannot be adjacent to either $u_1$ or $u_2$, as then $D$ and $u_1,u_2$ form a $P_4$, and $u_3 u_4 u_5 u_6$ form a $P_4$, and $u_7 u_8 u_9$ form a $P_3$. Analogously, $D$ cannot be adjacent to either $u_8$ or $u_9$. If $D$ is adjacent to $u_4$ or $u_6$, $D$ together with the adjacency and $u_5$ forms a $P_4$, and the rest of $P$ contains a $P_4 + P_3$. Similarly, if $D$ is adjacent to $u_5$, it forms a $P_4$ with $u_5, u_6$, and $u_7 u_8 u_9$ is a $P_3$ and $u_1 u_2 u_3 u_4$ is a $P_4$. Hence, $D$ can only be adjacent to $u_3$ or $u_7$.
	
If one of $d_1, d_2$ has $u_3$ as a neighbour, say $d_1$, and the other has $u_7$ as a neighbour, then $u_1 u_2 u_3 d_1 + u_4 u_5 u_6 + d_2 u_7 u_8 u_9$ form a $2P_4 + P_3$. So, either $d_1$ and $d_2$ are adjacent to only one of $u_3$ and $u_7$, or only one of $d_1, d_2$ has neighbours on $P$. But then there exists a cut vertex, which contradicts $G$ being 2-connected. We conclude that such a component $D$ of size 2 cannot exist in $G$.

It follows that every component in $G - P$ is of size 1, and $P$ forms a vertex cover of size 9. By Theorem~\ref{t-vc} we can solve {\sc Steiner Forest} in polynomial time.
\end{proof}

\noindent
{\it Proof of polynomial part of Theorem~\ref{t-main}.} 
Combine Lemma~\ref{l-sp2} with each of the others lemmas in this section.

\section{Conclusions}\label{s-con}

The aim of this paper was to increase our understanding of the complexity of {\sc Steiner Forest} and more generally, C23-problems, that is, graph problems that are not only \NP-complete on subcubic graphs (C2) and under edge division of subcubic graphs (C3), but also \NP-complete for graphs of bounded treewidth (not C1). Therefore, we studied {\sc Steiner Forest} for $H$-subgraph-free graphs. We significantly narrowed the number of open cases, thereby proving a number of boundary cases (as can be seen in Theorem~\ref{t-main}).
However, we were not able to generalize Lemma~\ref{l-sp2} from $P_2$ to $P_3$:

\begin{open}
Let $H$ be a graph. Is {\sc Steiner Forest} polynomial-time solvable on $(H+P_3)$-subgraph-free graphs if it is polynomial-time solvable for $H$-subgraph-free graphs?
 \end{open}
 
\noindent
 An affirmative answer to this question would reduce the number of open cases in Theorem~\ref{t-main} to a finite number. However, this requires a polynomial-time algorithm for {\sc Steiner Forest} on graphs with a $2$-deletion set of size~$d$ for any constant~$d$. The question if such an algorithm exists turned out to be highly challenging.

One important attempt that we made to solve the above question is to reduce instances to highly structured instances. In particular, we are able to reduce it to the case where the vertices of the deletion set itself belong to different connected components of a minimum Steiner forest, and all vertices not in the deletion set are terminals. However, even solving such highly structured instances seems difficult. 
We managed to reduce it to a Constraint Satisfaction Problem (CSP). Interestingly, this CSP can be solved in polynomial time for $2$-deletion sets of size~$2$ (cf.~Lemma~\ref{l-2d2}). The same CSP is \NP-complete when we consider deletion sets of size~$3$. Unfortunately, our reduction is only one way, so this does not directly imply \NP-completeness of {\sc Steiner Forest} in this case. Still, this hints that the problem might actually be \NP-complete for $H=sP_3$ for some $s\geq 4$.

It would also be interesting to know whether we can reduce the running time of our \FPT\ algorithm by the vertex cover number $\vc$ of a graph to $2^{O(\vc)} n^{O(1)}$.

We now discuss the complexity of {\sc Steiner Forest} on $H$-subgraph-free graphs versus that of {\sc Subgraph Isomorphism}. Bodlaender et al.~\cite{BHKKOO20} showed that if $H$ is connected and $H \not= P_5$, {\sc Subgraph Isomorphism} is polynomial-time solvable on $H$-subgraph-free graphs if $H$ is a subgraph of $P_4$ and \NP-complete otherwise. Observe that, following Theorem~\ref{t-main}, {\sc Steiner Forest} is solvable in polynomial time on $H$-subgraph-free graphs for a strictly larger set of connected graphs $H$. On the other hand, Bodlaender et al.\ also showed that {\sc Subgraph Isomorphism} is (randomized) polynomial-time solvable on $3P_4$-subgraph-free graphs, while {\sc Steiner Forest} is \NP-complete in this case (cf.~Theorem~\ref{t-main}). Hence, the two problems are incomparable.

Finally, the C123 problem {\sc Steiner Tree} is also classified with respect to the induced subgraph relation: it is polynomial-time solvable for $H$-free graphs if $H\ssi sP_1+P_4$ for some $s\geq 0$ and \NP-complete otherwise~\cite{BBJPPL21}. The hardness part of this result immediately carries over to {\sc Steiner Forest}. However, we do not know the complexity of {\sc Steiner Forest} on $(sP_1+P_4)$-free graphs. 

\medskip
\noindent
{\it Acknowledgments.} We thank Daniel Lokshtanov for pointing out a possible relationship between {\sc Steiner Forest} and CSP, as discussed in Section~\ref{s-con}.

\end{document}